\documentclass[11pt]{amsart}
\usepackage{amsmath,amsfonts,amssymb,amsthm}
\usepackage{enumerate}
\usepackage{esint}
\usepackage[pagebackref,colorlinks,citecolor=blue,linkcolor=blue]{hyperref}
\usepackage[dvipsnames]{xcolor}
\usepackage{mathtools}  

\usepackage{comment}

\usepackage[margin=1.5in]{geometry}

\numberwithin{equation}{section}
\theoremstyle{plain}
\newtheorem{theorem}[equation]{Theorem}
\newtheorem{corollary}[equation]{Corollary}
\newtheorem{lemma}[equation]{Lemma}

\theoremstyle{definition}
\newtheorem{definition}[equation]{Definition}

\newtheorem{example}[equation]{Example}
\newtheorem{remark}[equation]{Remark}

\newtheorem*{openproblem}{Open Problem}

\numberwithin{equation}{section}

\newcommand{\Z}{{\mathbb Z}}
\newcommand{\R}{{\mathbb R}}
\newcommand{\N}{{\mathbb N}}

\newcommand{\Om}{\Omega}

\providecommand{\vint}[1]{\mathchoice
          {\mathop{\vrule width 5pt height 3 pt depth -2.5pt
                  \kern -9pt \kern 1pt\intop}\nolimits_{\kern -5pt{#1}}}
          {\mathop{\vrule width 5pt height 3 pt depth -2.6pt
                  \kern -6pt \intop}\nolimits_{\kern -3pt{#1}}}
          {\mathop{\vrule width 5pt height 3 pt depth -2.6pt
                  \kern -6pt \intop}\nolimits_{\kern -3pt{#1}}}
          {\mathop{\vrule width 5pt height 3 pt depth -2.6pt
                  \kern -6pt \intop}\nolimits_{\kern -3pt{#1}}}}

\newcommand{\eps}{\varepsilon}
\newcommand{\loc}{\mathrm{loc}}

\newcommand{\BV}{\mathrm{BV}}

\newcommand{\ch}{\text{\raise 1.3pt \hbox{$\chi$}\kern-0.2pt}}

\DeclareMathOperator{\Mod}{Mod}

\DeclareMathOperator{\dist}{dist}
\DeclareMathOperator{\diam}{diam}

\DeclareMathOperator{\Lip}{Lip}

\DeclareMathOperator{\supp}{spt}

\begin{document}

\title[BV functions and Nonlocal functionals in metric measure spaces]{BV functions and nonlocal functionals\\
	 in metric measure spaces}
\author[P. Lahti]{Panu Lahti}
\address{Panu Lahti, Academy of Mathematics and Systems Science, Chinese Academy of Sciences,
Beijing 100190, PR China\\
Email: {\tt panulahti@amss.ac.cn}}

\author[A. Pinamonti]{Andrea Pinamonti}
\address{Andrea Pinamonti, 
Department of Mathematics, University of Trento, Trento 38123, Italy,
Email: {\tt andrea.pinamonti@unitn.it}}

\author[X. Zhou]{Xiaodan Zhou}
\address{Xiaodan Zhou, Analysis on Metric Spaces Unit, Okinawa Institute of Science and Technology G
raduate University, 1919-1 Tancha, Onna-son, 
Okinawa 904-0495, Japan\\
Email: {\tt xiaodan.zhou@oist.jp}}


\begin{abstract}
	We study the asymptotic behavior
	of three classes of nonlocal functionals in complete metric spaces equipped with a doubling
	measure and supporting a Poincar\'e inequality. We show that the limits of these nonlocal functionals
	are comparable to the variation $\Vert Df\Vert(\Om)$ or the Sobolev semi-norm
	$\int_\Om g_f^p\, d\mu$, which extends Euclidean results to metric measure spaces.
	In contrast to
	the classical setting, we also give an example to show that the limits are not always equal to
	the corresponding total variation even for Lipschitz functions.
\end{abstract}

\maketitle

\section{Introduction}

Consider a sequence $\{\rho_i\}_{i=1}^{\infty}$ of mollifiers in $L_{\loc}^1(0,\infty)$, $n\ge 1$,
for which
\begin{equation}\label{eq:intro rhoi conditions}
	\int_{0}^{\infty}\rho_i(r)r^{n-1}\,dr=1\ \textrm{ for all }i\in\N
	\quad\textrm{and}\quad
	\lim_{i\to \infty}\int_{\delta}^{\infty}\rho_i(r)r^{n-1}\,dr=0\quad\textrm{for all }\delta>0.
\end{equation}
Let $1\le p<\infty$.
For an open set $\Om\subset\R^n$ and function $f\in W_{\loc}^{1,p}(\Om)$, we define the Sobolev seminorm by
\[
|f|_{W^{1,p}(\Om)}:=\left(\int_{\Om}|\nabla f|^p\,dx\right)^{1/p},
\]
and if $f\notin W_{\loc}^{1,p}(\Om)$, then we let $|f|_{W^{1,p}(\Om)}=\infty$.
Given an open set $\Om\subset \R^n$ and  a function $f$ on $\Om$,
 we define
\[
	I_{p,i}(f,\Om):=\int_{\Om}\int_{\Om} \frac{|f(x)-f(y)|^p}{|x-y|^p}\rho_i(|x-y|)\,dx\,dy.
\]
We denote the energy by
\[
E_{p}(f,\Om)
:=\begin{cases}
	\Vert Df\Vert(\Om) &\textrm{when }p=1,\\
	|f|^p_{W^{1,p}(\Om)} &\textrm{when }1<p<\infty.
\end{cases}
\]

The so-called BBM formula, shown by Bourgain, Brezis, and Mironescu \cite[Theorem 3]{BBM} states that when $\Om\subset\R^n$ is a smooth, bounded domain and
$1<p<\infty$, then for every $f\in L^p(\Om)$ we have
\begin{equation}\label{eq: BBM}
    \lim_{i\to\infty}I_{p,i}(f,\Om)=K_{p,n}E_{p}(f,\Om),
\end{equation}
where $K_{p,n}>0$ is a constant depending only on $p$ and $n$. Later, D\'avila \cite{Dav} generalized this result to functions of bounded variation (BV functions)
$f$ and their variation measures $\Vert Df\Vert(\Om)$.
He showed that when $\Om\subset\R^n$ is a bounded domain with Lipschitz boundary,
then for every $f\in L^1(\Om)$ (we understand $\Vert Df\Vert(\Om)=\infty$ if $f\notin \BV(\Om)$), the above equality also holds for $p=1.$

Several different generalizations of these results in Euclidean spaces
have been considered e.g.
by Ponce \cite{Pon}, Leoni--Spector \cite{LS1,LS2}, Nguyen \cite{N1},
Brezis--Nguyen \cite{BN1,BN3, BN2, BN18}, 
Brezis--Van Schaftingen--Yung \cite{Bre2,Bre3,Bre4},
Nguyen--Pinamonti--Vecchi--Squassina \cite{NPVS,PVS1},    
Garofalo--Tralli \cite{GT, GT1}, Comi-Stefani \cite{CS,CS1,CS2},
Maalaoui--Pinamonti \cite{MP1}, Brena--Pasqualetto--Pinamonti \cite{BPP} and references therein.

On the other hand, in the past three decades, there has been significant progress in the study
of various aspects of first order analysis on metric measure spaces including the theory of first
order Sobolev spaces, functions of bounded variation and their relation to variational problems
and partial differential equations, see \cite{AT, Hj, HKST,Mir} and references therein.
In particular, in metric measure spaces $(X,d,\mu)$ with the measure being doubling and the space
supporting a Poincar\'e inequality (also known as PI spaces), fruitful results are obtained.

In \cite[Remark 6]{Bre}, Brezis raised a question about the relation between the BBM formula
and Sobolev spaces on metric measure spaces  \cite{AT,Hj,  KoSc}. Di Marino--Squassina \cite{DiMaS}
proved new characterizations of Sobolev and BV spaces in PI spaces in the spirit of the BBM formula.
A similar result was proved previously in Ahlfors-regular spaces by Munnier \cite{Mun}.
G\'orny \cite{Gor} and Han--Pinamonti \cite{HP}, respectively, studied the problem in certain
PI spaces that ``locally look like'' Euclidean spaces, or respectively finite-dimensional
Banach spaces or Carnot groups.

Given a metric measure space $(X,d,\mu)$ and an open set $\Om\subset X$, consider the following functional and energy:
\begin{equation}\label{eq:functional basic}
	I_{p,i}(f,\Om):=\int_{\Om}\int_{\Om} \frac{|f(x)-f(y)|^p}{d(x,y)^p}\rho_i(x,y)\,d\mu(x)\,d\mu(y),
\end{equation}
and
\[
E_{p}(f,\Om)
:=\begin{cases}
	\Vert Df\Vert(\Om) &\textrm{when }p=1,\\
	\int_{\Om}g_f^p\,d\mu &\textrm{when }1<p<\infty,
\end{cases}
\]
where $g_f$ is the so-called minimal $p$-weak upper gradient of $f$ which equals $|\nabla f|$ for Sobolev functions $f\in W^{1,p}(\Om)$ in the Euclidean space. See Section \ref{sec:definitions} for precise definition. If $f$ is not in $\BV(\Om)$ when $p=1$,
respectively not in $\widehat{N}^{1,p}(\Om)$ when $1<p<\infty$,  then we let $E_{p}(f,\Om)=\infty$. 

In our previous paper \cite{LPZ}, we studied the BBM formula for a more general class of mollifiers
$\rho_i$ which contain the mollifiers considered in \cite{DiMaS, Gor, HP} as special cases.
The explicit conditions will
be specified in Section \ref{sec:definitions}. Similar to the BBM formula \eqref{eq: BBM},
we obtained the following \cite[Theorem 1.1]{LPZ}:

\begin{theorem}\label{thm:main previous}
	Let $1\le p<\infty$, suppose $\mu$ is doubling, and let $\Om\subset X$ be open.
		If $\{\rho_i\}_{i=1}^{\infty}$ is a sequence of mollifiers satisfying condition
	\eqref{eq:rho hat minorize}, and $f\in L^p(\Om)$, then
	\[
		\liminf_{i\to\infty}I_{p,i}(f,\Om)\ge C_1 E_{p}(f,\Om).
	\]
	If  $X$ supports a $(p,p)$-Poincar\'e inequality, $\{\rho_i\}_{i=1}^{\infty}$ is a sequence of mollifiers satisfying condition
	\eqref{eq:rho hat majorize},
	$\Om$ is a strong $p$-extension domain, and $f\in L^p(\Om)$, then
	\[
			 \limsup_{i\to\infty}I_{p,i}(f,\Om)
			\le C_2 E_{p}(f,\Om).
	\]
	Here $C_1\le C_2$ are constants that depend only on $p$,
	the doubling constant of the measure, the constants in the
	Poincar\'e inequality, and the constant $C_{\rho}$ associated with the mollifiers.
\end{theorem}

In general, contrary to Euclidean spaces, we cannot expect
$C_1=C_2$ in the above estimates,
as we demonstrate in \cite[Section 7]{LPZ}.

It is natural and interesting to ask whether such asymptotic behaviors of \eqref{eq:functional basic}
still hold for other convex or nonconvex functionals. The pursuit of such question is also driven
by applications. For example, Gilboa--Osher \cite{GiOs} consider nonlocal image functionals to
improve effectiveness in image denoising and reconstruction. One prototype functional considered
in \cite{GiOs} corresponds to a special case with $p=1$, $q=2$ of the nonlocal functionals 
\[
\int_{\Om}\left[\int_{\Om} \Big(\frac{|f(x)-f(y)|^p}{|x-y|^p}\Big)^q\rho_i(|x-y|)\,dx\right]^{1/q}\,dy
\]
studied by Leoni--Spector \cite{LS1,LS2}.
Leoni--Spector also obtain $\Gamma$-convergence of the above functional to a constant multiple of $\int_{\Om}|\nabla f|^p\,dx$ for all $f\in W^{1,p}(\Om)$ and $p\in (1,\infty)$ under appropriate assumptions on $q$ and $\rho_i$ \cite[Theorem 1.2]{LS2}.

Furthermore, Brezis--Nguyen \cite{BN1} study the asymptotic behavior of two
convex nonlocal functionals which converge formally to the total variation.
In \cite{BN2, BN4}, Brezis--Nguyen consider approximation of $\int_\Omega |\nabla u|^p$
with $p\ge 1$ for nonconvex nonlocal functionals motivated by questions arising in image processing.

Inspired by the work of Bourgain-Brezis-Mironescu \cite{BBM}, Leoni-Spector \cite{LS1,LS2},
and Brezis-Nguyen \cite{BN1, BN2, BN18, BN4}, the main goal of the current paper is to study
the asymptotic behavior of certain convex or nonconvex nonlocal functionals in metric measure spaces.
In particular, we focus on the following three functionals from the work Leoni-Spector \cite{LS1}
and Brezis--Nguyen \cite{BN1, BN2, BN18,  BN4}:

\begin{eqnarray}
&\nonumber \Psi_{p,i}(f,\Om):= & \left(\int_{\Om}\int_{\Om} \frac{|f(x)-f(y)|^{p+\eps_i}}{d(x,y)^{p+\eps_i}}\rho_i(x,y)\,d\mu(x)\,d\mu(y)
\right)^{p/(p+\eps_i)},\quad\eps_i\searrow 0,\\
&\nonumber \Phi_{p,q,i}(f,\Om):= & \int_{\Om}\left[\int_{\Om} \left(\frac{|f(x)-f(y)|^p}{d(x,y)^p}\right)^q\rho_i(x,y)\,d\mu(x)\right]^{1/q}\,d\mu(y),\\
&\nonumber  \Lambda_{p,\delta}(f,\Om):= & \int_{\Om}\int_{\Om} \frac{\delta^p\varphi(|f(x)-f(y)|/\delta)}{\mu(B(x,d(x,y)))d(x,y)^p} 
\,d\mu(x)\,d\mu(y),\quad\delta>0.
\end{eqnarray}


The conditions on the mollifiers $\rho_i$ and $\varphi\colon [0,\infty)\to [0,\infty)$ will be made explicit in Section \ref{sec:definitions} and some typical examples of $\rho_i$ and $\varphi$ will be given in Section \ref{sec:examples}. In particular, the conditions posed on $\varphi$ 
 imply that $\Lambda_{p,\delta}$ is generally a nonconvex functional. We obtain the following theorems of the asymptotic behaviors on the above three functionals respectively. 


\begin{theorem}\label{thm:main eps}
	Let $1\le p<\infty$ and suppose $\mu$ is doubling, suppose
	$\{\rho_i\}_{i=1}^{\infty}$ is a sequence of mollifiers satisfying condition
	\eqref{eq:rho hat minorize}, \eqref{eq:rho hat majorize},
	and let $\Om\subset X$ be open and bounded.
	
	If $f\in L^p(\Om)$, then
	\[
	C_1' E_p(f,\Om)
	\le \liminf_{i\to\infty}\Psi_{p,i}(f,\Om).
	\]
	
	If, on the other hand, 
$X$ supports a $(1,p)$--Poincar\'e inequality and there is $q\in (p,\infty)$ such that
$\Om\subset X$ is a strong $p$-extension and strong $q$-extension domain, and
$f\in \widehat{N}^{1,q}(\Om)$, then
	\[
			\limsup_{i\to\infty}\Psi_{p,i}(f,\Om)
			\le C_2' E_p(f,\Om).
	\]
	Here $C_1'\le C_2'$ are constants that depend only on $p$,
	the doubling constant of the measure, the constants in the
	Poincar\'e inequality, and the constant $C_{\rho}$ associated with the mollifiers.
\end{theorem}

\begin{theorem}\label{thm:main p}
	Let $1\le p<\infty$, $1<q<\infty$, suppose $\mu$ is doubling,
	$\{\rho_i\}_{i=1}^{\infty}$ is a sequence of mollifiers satisfying condition
	\eqref{eq:rho hat minorize} or the same condition with $p$ replaced by $q$,
	 as well as \eqref{eq:rho hat majorize},
	and let $\Om\subset X$ be open and bounded.
	
	If $f\in L^p(\Om)$, then
	\[
	C_1'' E_{p}(f,\Om)
	\le \liminf_{i\to\infty}\Phi_{p,q,i}(f,\Om).
	\]
	
	If  $X$ supports a $(1,p)$--Poincar\'e inequality, $\Om\subset X$ is a bounded $pq$-extension domain, and $f\in \widehat{N}^{1,pq}(\Om)$, then
	\[
			\limsup_{i\to\infty}\Phi_{p,q,i}(f,\Om)
			\le C_2'' E_{p}(f,\Om).
	\]
	Here $C_1''\le C_2''$ are constants that depend only on $p$, $q$,
	the doubling constant of the measure, the constants in the
	Poincar\'e inequality, and the constant $C_{\rho}$ associated with the mollifiers.
\end{theorem}

\begin{theorem}\label{thm:main phi}
	Let $1\le p<\infty$ and suppose $\varphi\colon [0,\infty)\to [0,\infty)$ satisfies assumptions
	\eqref{eq:varphi assumption increasing}--\eqref{eq:varphi integral assumption}, and
	suppose $\Om\subset X$ is open and bounded.
	
	If $\mu$ is Ahlfors $Q$-regular with $1<Q<\infty$
	and $f\in L^p(\Om)$, then
	\[
	C_1''' E_{p}(f,\Om)
	\le \limsup_{\delta\to 0}\Lambda_{p,\delta}(f,\Om)
	\]
	for some constant $C_1'''$ depending only on $p$, the Ahlfors regularity constants,
	and the constant $C_{\varphi}$ associated with the functional $\Lambda_{p,\delta}$.
	
	If, on the other hand, $\mu$ is doubling, $X$ supports a $(1,p)$--Poincar\'e inequality,
	and $f\in \widehat{N}^{1,1}(X)\cap \widehat{N}^{1,q}(X)$ for some $1<q<\infty$ in the case $p=1$,
	and  
	and $f\in \widehat{N}^{1,p}(X)$ in the case $1<p<\infty$, then
	\[
	\limsup_{\delta\to 0}\Lambda_{p,\delta}(f,\Om)\le C_2''' E_{f,p}(\Om)
	\]
	for some constant $C_2'''$ depending only on $p$, the doubling constant,
	the constants in the Poincar\'e inequality, and the constant $C_{\varphi}$.
\end{theorem}

\begin{remark}
    In general, one cannot have equality of the constants.
    In particular, in Theorem \ref{thm:main p} one might hope for the equality $C_1''=C_2''$.
    Brezis-Nguyen \cite[Proposition 12]{BN1} give an example where
    $\Phi_{1,q,i}(f,\Om)$ defined with radial mollifiers satisfying \eqref{eq:intro rhoi conditions}
    fail to converge to $K_{1,1}E_1(f,\Om)$ for a Sobolev function $f\in W^{1,1}(\Om)$ on $\Om=(-1/2,1/2).$
    If $f$ is further assumed to be in $W^{1,q}(\Om)$ for some $q>1$ on an open, bounded and smooth
    domain in $\mathbb{R}^n$, such convergence holds \cite[Proposition 9]{BN1}. However,
    we will give an explicit example (Example \ref{ex:fat Cantor}) showing that convergence
    to a constant times $E_1(f,\Om)$ can fail
    even for a Lipschitz function on metric measure spaces. 
\end{remark}

\begin{remark}
    It is worth pointing out that in some cases the convergence of the above functionals in the
    Euclidean space is not yet fully understood, for example $\Phi_{p,q,i}(f,\Om)$ for a 
    general $1\le p<\infty$ \cite[Section 3]{BN1}.
    On the other hand, $\Gamma$-convergence of the functionals seems to be more robust and clear.
    In Euclidean spaces, the $\Gamma$-convergence of  the functionals
    $\Phi_{1,q,i}(\cdot,\Om)$, $\Psi_{1,i}(\cdot,\Om)$, $\Lambda_{1,\delta}(\cdot,\Om)$ is proved
    in \cite{BN1, BN18}. It would be interesting to investigate the $\Gamma$-convergence of
    such functionals in metric measure spaces as well. 
\end{remark}

The paper is organized as follows. We will review some definitions and introduce the conditions of the mollifiers $\rho_i$ and the function $\varphi$ in Section \ref{sec:definitions}. Then we give some examples of $\rho_i$, $\varphi$ and the related corollaries by applying such mollifiers in the main theorems in Section
\ref{sec:examples}. We prove the lower bounds of our three main theorems 
in Section \ref{sec:lower bounds},
and the upper bounds in Section \ref{sec:upper bounds}.
In Section \ref{sec:counterexample} we give an example demonstrating that
we do not generally have $C_1''=C_2''$.

\section{Notation and definitions}\label{sec:definitions}

Throughout this paper, we work in a complete and connected
metric measure space $(X,d,\mu)$ equipped with a metric $d$ and
a Borel regular outer measure $\mu$ satisfying
a doubling property, meaning that
there exists a constant $C_d\ge 1$ such that
\[
0<\mu(B(x,2r))\le C_d\mu(B(x,r))<\infty
\]
for every ball $B(x,r):=\{y\in X\colon\,d(y,x)<r\}$, with $x\in X$, $r>0$.
We assume that $X$ consists of at least two points, that is, $\diam X>0$.
As a complete metric space equipped with a doubling measure, $X$ is proper, meaning that
closed and bounded sets are compact.
By \cite[Corollary 3.8]{BB}, we know that for every $x\in X$ and every $0<r\le R<\tfrac 12 \diam X$, we have
\begin{equation}\label{eq:upper mass bound}
\frac{\mu(B(x,r))}{\mu(B(x,R))}\le C_0\left(\frac{r}{R}\right)^{\sigma}
\end{equation}
for constants $C_0>0$ and $0<\sigma<1$ depending only on $C_d$.
It follows that
$\mu(\{x\})=0$ for every $x\in X$. 

Sometimes we will make the stronger assumption that $\mu$ is Ahlfors regular.
We say that $\mu$ is Ahlfors $Q$-regular with $1<Q<\infty$ if for every $x\in X$ and
every $0<r<2\diam X$, we have
\[
C_A^{-1} r^Q\le \mu(B(x,r))\le C_A r^Q,
\]
for constant $C_A>0$ depending only on $Q$.
Obviously Ahlfors regularity implies doubling.
 
By a curve we mean a rectifiable continuous mapping from a compact interval of the real line into $X$.
The length of a curve $\gamma$
is denoted by $\ell_{\gamma}$. We will assume every curve to be parametrized
by arc-length, which can always be done (see e.g. \cite[Theorem~3.2]{Hj}).
A nonnegative Borel function $g$ on $X$ is an upper gradient 
of a function $f\colon X\to [-\infty,\infty]$
if for all nonconstant curves $\gamma\colon [0,\ell_{\gamma}]\to X$, we have
\begin{equation}\label{eq:definition of upper gradient}
	|f(x)-f(y)|\le \int_{\gamma} g\,ds:=\int_0^{\ell_{\gamma}} g(\gamma(s))\,ds,
\end{equation}
where $x$ and $y$ are the end points of $\gamma$.
We interpret $|f(x)-f(y)|=\infty$ whenever  
at least one of $|f(x)|$, $|f(y)|$ is infinite.
Upper gradients were originally introduced in \cite{HK}.

Let $1\le p<\infty$.
The $p$-modulus of a family of curves $\Gamma$ is defined by
\[
\Mod_{p}(\Gamma):=\inf\int_{X}\rho^p\, d\mu,
\]
where the infimum is taken over all nonnegative Borel functions $\rho$
such that $\int_{\gamma}\rho\,ds\ge 1$ for every curve $\gamma\in\Gamma$.
A property is said to hold for $p$-almost every curve
if it fails only for a curve family with zero $p$-modulus. 
If $g$ is a nonnegative $\mu$-measurable function on $X$
and (\ref{eq:definition of upper gradient}) holds for $p$-almost every curve,
we say that $g$ is a $p$-weak upper gradient of $f$. 
By only considering curves $\gamma$ in a set $A\subset X$,
we can talk about a function $g$ being a ($p$-weak) upper gradient of $u$ in $A$.\label{curve discussion}

We always let $\Om$ denote an open subset of $X$.
We define the Newton-Sobolev space $N^{1,p}(\Om)$ to consist of those functions $f\in L^p(\Om)$ for which there
exists  a $p$-weak upper gradient $g\in L^p(\Om)$ of $f$ in $\Om$.
This space was first introduced in \cite{S}.
We write $f\in N^{1,p}_{\loc}(\Om)$ if for every $x\in \Om$ there exists $r>0$ such that
$f\in N^{1,p}(B(x,r))$; other local function spaces are defined analogously.
For every $f\in N^{1,p}_{\loc}(\Om)$ there exists a minimal $p$-weak
upper gradient of $f$ in $\Om$, denoted by $g_f$, satisfying $g_f\le g$ 
$\mu$-almost everywhere (a.e.)
in $\Om$ for every $p$-weak upper gradient $g\in L_{\loc}^{p}(\Om)$
of $f$ in $\Om$, see \cite[Theorem 2.25]{BB}.

Note that Newton-Sobolev functions are understood to be defined at every $x\in \Om$, whereas the functionals
that we consider are not affected by perturbations of $f$ in a set of zero $\mu$-measure. For this reason, we also define
\[
\widehat{N}^{1,p}(\Om):=\{f\colon \Om\to [-\infty,\infty]\colon f=h \ \mu\textrm{-a.e. in }\Om\textrm{ for some }h\in N^{1,p}(\Om)\}.
\]
For every $f\in \widehat{N}^{1,p}(\Om)$, we can also define
$g_f:=g_h$, where $g_h$ is the minimal $p$-weak upper gradient of any $h$ as above in $\Om$;
this is well defined $\mu$-a.e. in $\Om$ by \cite[Corollary 1.49, Proposition 1.59]{BB}.

Next we give Mazur's lemma, see e.g. \cite[Theorem 3.12]{Rud}.

\begin{theorem}\label{thm:Mazur lemma}
	Let $\{g_i\}_{i=1}^{\infty}$ be a sequence
	with $g_i\to g$ weakly in $L^1(\Om)$.
	Then there exist convex combinations $\widehat{g}_i:=\sum_{j=1}^{N_i}a_{i,j}g_j$,
	for some $N_i\in\N$,
	such that $\widehat{g}_i\to g$ in $L^1(\Om)$.
\end{theorem}

By convex combinations we mean that the numbers $a_{i,j}$ are nonnegative
and that $\sum_{j=1}^{N_i}a_{i,j}=1$ for every $i\in\N$.

Let also $1\le q<\infty$.
We say that $X$ supports a $(q,p)$-Poincar\'e inequality
if there exist constants $C_P>0$ and $\lambda \ge 1$ such that for every
ball $B(x,r)$, every $f\in L^1(X)$,
and every upper gradient $g$ of $f$, we have
\begin{equation}\label{eq:pp Poincare}
	\left(\vint{B(x,r)}|f-f_{B(x,r)}|^q\, d\mu\right)^{1/q}
	\le C_P r \left(\vint{B(x,\lambda r)}g^p\,d\mu\right)^{1/p},
\end{equation}
where, as usual,
\[
f_{B(x,r)}:=\vint{B(x,r)}f\,d\mu :=\frac 1{\mu(B(x,r))}\int_{B(x,r)}f\,d\mu.
\]

Next we define functions of bounded variation.
Given an open set $\Om\subset X$ and a function $f\in L^1_{\loc}(\Om)$,
we define the total variation of $f$ in $\Om$ by
\[
\Vert Df \Vert(\Om):=\inf\left\{\liminf_{i\to\infty}\int_\Om g_{f_i}\,d\mu\colon\, f_i\in N^{1,1}_{\loc}(\Om),\, f_i\to f\textrm{ in } L^1_{\loc}(\Om)\right\},
\]
where each $g_{f_i}$ is the minimal $1$-weak upper gradient of $f_i$
in $\Om$.
We say that a function $f\in L^1(\Om)$ is of bounded variation, 
and denote $f\in\BV(\Om)$, if $\Vert Df\Vert (\Om)<\infty$.
For an arbitrary set $A\subset X$, we define
\[
\Vert Df \Vert (A):=\inf\{\Vert Df\Vert (W)\colon\, A\subset W,\,W\subset X
\text{ is open}\}.
\]
If $f\in L^1_{\loc}(\Om)$,
then $\Vert Df\Vert(\cdot)$ is
a Borel regular outer measure on $\Omega$ by \cite[Theorem 3.4]{Mir}.

For a function $u$ defined on an open set $\Om\subset X$,
we abbreviate super-level sets in the form
\[
\{u>t\}:=\{x\in \Om\colon u(x)>t\},\quad t\in\R.
\]
The following coarea formula is given in \cite[Proposition 4.2]{Mir}:
if $\Omega\subset X$ is open and $u\in \BV(\Om)$, then
\begin{equation}\label{eq:coarea}
	\|Du\|(\Omega)=\int_{-\infty}^{\infty}P(\{u>t\},\Omega)\,dt.
\end{equation}
If $f\in N^{1,1}(\Om)$, then
\begin{equation}\label{eq:from BV to Sobolev}
	\int_{\Om}g_{f}\,d\mu \le C_*\Vert Df\Vert(\Om),
\end{equation}
where $g_f$ is as usual the minimal $1$-weak upper gradient of $f$ in $\Om$, and
$C_*\ge 1$ is a constant depending only on the doubling constant $C_d$ and the constants in the Poincar\'e inequality $C_P, \lambda$; 
see \cite[Remark 4.7]{HKLL}.

\begin{definition}\label{def:strong BV extension}
	We say that an open set $\Om\subset X$ is a strong $p$-extension domain if
	\begin{itemize}
		\item in the case $p=1$, for every $f\in \BV(\Om)$
		there exists an extension $F\in \BV(X)$ and $\Vert DF\Vert(\partial\Om)=0$;
		\item in the case $1<p<\infty$, for every $f\in \widehat{N}^{1,p}(\Om)$ there exists
		an extension $F\in \widehat{N}^{1,p}(X)$
		and $\int_{\partial \Om}g_F^p\,d\mu=0$.
	\end{itemize}
A $p$-extension domain is defined similarly, but we omit the conditions
$\Vert DF\Vert(\partial\Om)=0$ resp. $\int_{\partial \Om}g_F^p\,d\mu=0$.
\end{definition}

For example, in Euclidean spaces, a bounded domain with a Lipschitz boundary is a strong $p$-extension domain for all $1\le p<\infty$,
see e.g. \cite[Proposition 3.21]{AFP}.

The Hardy--Littlewood maximal function of a nonnnegative function $g\in L_{\loc}^1(X)$ is defined by
\begin{equation}\label{eq:maximal function}
\mathcal M g(x):=\sup_{r>0}\,\vint{B(x,r)}g\,d\mu.
\end{equation}
Given $R>0$, we define the restricted maximal function $\mathcal M_Rg(x)$ in a similar way, but
we take the supremum over radii $0<r\le R$.

Given $U\subset X$ and $\delta>0$, we denote
\begin{equation}\label{eq:neighborhood notation}
	U(\delta):=\{x\in X\colon d(x,U)<\delta\}.
\end{equation}

Let $f$ be a function defined on $\Om$. We define
\begin{equation}\label{eq:Lip r}
\Lip_r f(x):=\sup_{y\in \Om\cap B(x,r)}\frac{|f(y)-f(x)|}{r},
\quad x\in \Om,\quad r>0,
\end{equation}
and
\begin{equation}\label{eq:Lip}
\Lip f(x):=\limsup_{r\to 0}\Lip_r(x).
\end{equation}

Now we describe the mollifiers that we will use in much of the paper.
Let $1\le p<\infty$.
We will consider a sequence of nonnegative $X\times X$-measurable functions
$\{\rho_i\}_{i=1}^{\infty}$ and a fixed constant
$1\le C_{\rho}<\infty$ satisfying the following conditions:
\begin{enumerate}[(1)]\label{rho conditions}
	\item For every $i\in\N$ and for every $x,y\in X$, we have 
	\begin{equation}\label{eq:rho hat minorize}
		\rho_i(x,y)
		\ge
		C_\rho^{-1} \frac{d(x,y)^p}{r_i^{p}}\frac{\ch_{B(y,r_i)}(x)}{\mu(B(y,r_i))},
	\end{equation}
	where $r_i\searrow 0$.

	\item For every $i\in\N$ and every $x,y\in X$, we have
	\begin{equation}\label{eq:rho hat majorize}
		\rho_i(x,y)
		\le \sum_{j\in\Z} d_{i,j}\frac{\ch_{B(y,2^{j+1})
				\setminus B(y,2^{j})}(x)}{\mu(B(y,2^{j+1}))}
	\end{equation}
	for numbers $d_{i,j}\ge 0$ for which $\sum_{j\in\Z} d_{i,j}\le C_\rho$
	and $\lim_{i\to\infty}\sum_{j\ge M} d_{i,j}=0$ for all $M\in\Z$.
\end{enumerate}

\begin{remark}
The assumptions \eqref{eq:intro rhoi conditions} are ubiquitous in the literature,
but depending on the functional under consideration, some additional conditions are needed. For example,  the convergence of
$\Psi_{1,i}(f,\Om)$ to $E_1(f,\Om)$ in a smooth bounded domain in the Euclidean space 
has been verified only for a special choice of mollifiers \cite[Proposition 2]{BN1}.
In our setting,
we do not have access to certain Euclidean techniques, especially the Taylor approximation.
Our formulation of the assumptions \eqref{eq:rho hat minorize} and \eqref{eq:rho hat majorize}
is informed by these facts.
These assumptions are in any case satisfied by certain typical and important choices of the mollifiers
$\rho_i$, as we will see in Section  \ref{sec:examples}.
\end{remark}

The definition of the functional $\Lambda_{p,\delta}(f,\Om)$ involves a function
$\varphi\colon [0,\infty)\to [0,\infty)$. Let $1\le p<\infty$.
Given $b>0$, we consider the assumptions (we understand ``increasing'' in the non-strict sense)
\begin{equation}\label{eq:varphi assumption increasing}
	\varphi(t)\quad\textrm{is increasing},
\end{equation}
\begin{equation}\label{eq:varphi assumption bounded}
	\varphi(t)\le b\quad\textrm{for }0\le t<\infty;
\end{equation}
and
\begin{equation}\label{eq:varphi integral assumption}
C_\varphi^{-1}\le 	\int_0^{\infty}\varphi(t)t^{-1-p}\,dt\le C_\varphi
\quad\textrm{for some }1\le C_{\varphi}<\infty.
\end{equation}

Our standing assumptions are the following; note that we do \emph{not} always
assume that $\mu$ is Ahlfors regular or that $X$ satisfies a Poincar\'e inequality.\\

\emph{Throughout the paper, $(X,d,\mu)$ 
	is a complete and connected
	metric space 
	equipped with a doubling Borel regular outer measure $\mu$, with $\diam X>0$. We always assume $\Om\subset X$ to be an open set.}

\section{Mollifiers and implications}\label{sec:examples}

In this section we consider three important examples
of mollifiers $\rho_i$ satisfying
conditions \eqref{eq:rho hat minorize} and \eqref{eq:rho hat majorize}. The first mollifier is
investigated carefully in our previous work \cite{LPZ} but we add it here to obtain some results
needed in our later proofs. We apply the second and third mollifiers in Theorem \ref{thm:main p}
with the choice $p=1$ to give some interesting corollaries.  We also give a simple choice of
nonconvex function $\varphi$ in Theorem \ref{thm:main phi} at the end of this section and compare
the result with that in \cite{BP}.

First we recall one mollifier from \cite{LPZ} which gives important corollaries needed later. 
Consider 
\[
\rho_i(x,y):=(1-s_i)\frac{1}{d(x,y)^{p(s_i-1)}\mu(B(y,d(x,y)))}, \quad x,y\in X,
\]
where $s_i\nearrow 1$ as $i\to \infty.$ One can verify that it satisfies \eqref{eq:rho hat minorize} and as a result we obtain the following corollary, see \cite[Corollary 6.1]{LPZ} for detailed proof. 

\begin{corollary}\label{cor:DiMarino}
	Let $1\le p<\infty$ and let $f\in L^p(\Om)$. Then
	\begin{equation}
		\liminf_{s\nearrow 1}(1-s)\int_{\Om}\int_{\Om}\frac{|f(x)-f(y)|^p}{d(x,y)^{ps}\mu(B(y,d(x,y)))}\,d\mu(x)\,d\mu(y)
			\ge C^{-1}E_{p}(f,\Om)
	\end{equation}
	for some constant $C\ge 1$ depending only on the doubling constant of the measure.
\end{corollary}
\begin{remark}
    Note that in \cite[Corollary 6.1]{LPZ} it is also assumed that $\Om$ is a strong $p$-extension domain
and that $X$ supports a $(p,p)$-Poincar\'e inequality,
but these are only needed for the upper bound given in that corollary.
\end{remark}
The following is now an easy consequence of Corollary \ref{cor:DiMarino}.
\begin{corollary}\label{cor:epsilon consequence}
	Let $1\le p<\infty$ and suppose that $\mu$ is Ahlfors $Q$-regular
	and that $\Om$ is bounded. Let $f\in L^p(\Om)$. Then
	\[
	\liminf_{\eps\searrow 0}\int_{\Om}\int_{\Om}\frac{\eps |f(x)-f(y)|^{p+\eps}}{d(x,y)^{Q+p}}
	\,d\mu(x)\,d\mu(y)\ge C^{-1} E_{p}(f,\Om)
	\]
	for some constant $C\ge 1$ depending only on $p$ and the Ahlfors regularity constants of the measure.
\end{corollary}
\begin{proof}
	First note that
	\begin{equation}\label{eq:epsilon limit}
	\lim_{\eps\searrow 0}\eps^{p/(p+\eps)}\frac{{p(p+\eps)}}{\eps(p+Q)}
	=\frac{p^2}{p+Q}\lim_{\eps\searrow 0}\eps^{-\eps/(p+\eps)}=\frac{p^2}{p+Q}.
	\end{equation}
	By H\"older's inequality, we have	
	\begin{align*}
		&\liminf_{\eps\searrow 0}\int_{\Om}\int_{\Om}\frac{\eps |f(x)-f(y)|^{p+\eps}}{d(x,y)^{Q+p}}
		\,d\mu(x)\,d\mu(y)\\
		& \ge \liminf_{\eps\searrow 0}\eps\mu(\Om)^{-2\eps/p}\left(\int_{\Om}\int_{\Om}
		\frac{ |f(x)-f(y)|^p}{d(x,y)^{(Q+p)p/(p+\eps)}}\,d\mu(x)\,d\mu(y)\right)^{(p+\eps)/p}\\
		& = \liminf_{\eps\searrow 0}\Bigg(\eps^{p/(p+\eps)}\frac{{p(p+\eps)}}{\eps(p+Q)}
		\left(1-\frac{p^2-Q\eps}{p(p+\eps)}\right)\\
		&\qquad \qquad \times \int_{\Om}\int_{\Om}
		\frac{ |f(x)-f(y)|^p}{d(x,y)^{Q}d(x,y)^{(p^2-Q\eps)/(p+\eps)}} \,d\mu(x)\,d\mu(y)\Bigg)^{(p+\eps)/p}\\
		& \ge\frac{p^2}{p+Q}C^{-1}E_{p}(f,\Om)
	\end{align*}
by \eqref{eq:epsilon limit} and Corollary \ref{cor:DiMarino}.
\end{proof}
The second mollifier we consider here is defined as

	\[
	\rho_{i}(x,y)
	=\frac{d(x,y)^q}{r_i^{q}}\frac{\ch_{B(y,r_i)}(x)}{\mu(B(y,r_i))},
	\]
	where $r_i\searrow 0$ as $i\to\infty$. 

It is not hard to check that
\eqref{eq:rho hat minorize} holds. We next verify that it satisfies \eqref{eq:rho hat majorize}.
	Let
	\[
	d_{i,j}:=\frac{2^{(j+1)q}}{r_i^{q}}C_d 
	\]
	for $j\le \log_2 r_i$, and $d_{i,j}=0$ otherwise. Now
	\[
	d_{i,j}\ge \sup_{y\in X} \frac{2^{(j+1)q}}{r_i^{q}}\frac{\mu(B(y,2^{j+1}))}{ \mu(B(y,r_i))}
	\]
	for $j\le \log_2 r_i$, and so
	for every $x,y\in X$, we have
	\[
	\rho_{i}(x,y)
	\le \sum_{j\in\Z}d_{i,j}\frac{\ch_{B(y,2^{j+1})
			\setminus B(y,2^{j})}(x)}{\mu(B(y,2^{j+1}))}.
	\]
	Moreover,
	\begin{align*}
		\sum_{j\in\Z}d_{i,j}
		= C_d r_i^{-q} \sum_{j\le \log_2 r_i}  2^{(j+1)q}
		\le  2^{q+1} C_d,
	\end{align*}
	and
	\begin{align*}
	\sum_{j\ge M}d_{i,j}=0
	\end{align*}
	when $\log_2 r_i<M$, so that $\lim_{i\to\infty}\sum_{j\ge M}d_{i,j}=0$ for every $M\in\Z$.
 
 The following corollary follows from applying Theorem \ref{thm:main p} with the choices $p=1$ and the above mollifier. 
\begin{corollary}\label{cor:Gorny}
	Let $1<q<\infty$ and suppose that
	$X$ supports a $(1,1)$--Poincar\'e inequality.
	If $f\in L^q(\Om)$,
	then
	\[
	C^{-1} \Vert Df\Vert(\Om)
	\le \liminf_{r\searrow 0}\int_\Om
	\left[\,\frac{1}{\mu(B(y,r))}\int_{B(y,r)\cap \Om}\frac{|f(x)-f(y)|^q}{r^q}\,d\mu(x)\right]^{1/q}\,d\mu(y).
	\]
	If $\Om$ is a bounded $q$-extension domain and $f\in \widehat{N}^{1,q}(\Om)$, then
	\[
	\limsup_{r\searrow 0}\int_\Om 
	\left[\,\frac{1}{\mu(B(y,r))}\int_{B(y,r)\cap\Om}\frac{|f(x)-f(y)|^q}{r^q}\,d\mu(x)
	\right]^{1/q}\,d\mu(y)
	\le C \Vert Df\Vert(\Om).
	\]
	Here $C\ge 1$ is a constant depending only on $q$, the doubling constant of the measure,
	and the constants in the Poincar\'e inequality.
\end{corollary}

The third mollifier, studied in the Euclidean setting e.g. by Brezis
\cite[Eq. (45)]{Bre}, is simple and natural. 
Consider
\[
	\rho_{i}(x,y)=\frac{\ch_{B(y,r_i)}(x)}{\mu(B(y,r_i))},
	\]
	where $r_i\searrow 0$ as $i\to\infty$.
	Again \eqref{eq:rho hat minorize} can be verified easily.
	We check \eqref{eq:rho hat majorize}. We can assume that $r_i<\min\{1,\diam X/4\}$ for all $i\in\N$.
	Letting
	\[
	d_{i,j}:=C_0 \left(\frac{2^{j+1}}{r_i}\right)^{\sigma},
	\]
	for $j\le \log_2 r_i$ and $d_{i,j}=0$ otherwise, by \eqref{eq:upper mass bound} we have
	for all $j\le \log_2 r_i$ that
	\[
	d_{i,j}\ge \frac{\mu(B(y,2^{j+1}))}{ \mu(B(y,r_i))}
	\quad\textrm{for all }y\in X.
	\]
	Then for every $x,y\in X$, we have
	\[
	\rho_{i}(x,y)
	\le \sum_{j\in\Z}d_{i,j}\frac{\ch_{B(y,2^{j+1})\setminus B(y,2^{j})}(x)}{\mu(B(y,2^{j+1}))}.
	\]
	Finally,
	\[
	\sum_{j\in\Z} d_{i,j}
	= C_0\sum_{j\le \log_2 r_i} \left(\frac{2^{j+1}}{r_i}\right)^{\sigma}
	\le C,
	\]
	where $C$ depends only on $C_0$ and $\sigma$, and thus in fact only on the doubling 
	constant $C_d$.

 Similar to the case of second mollifier, the following corollary follows from applying
 Theorem \ref{thm:main p} with the choice $p=1$.

\begin{corollary}\label{cor:Brezis}
	Let $1<q<\infty$ and suppose that
	$X$ supports a $(1,1)$--Poincar\'e inequality.
	If
	$f\in L^q(\Om)$, then
	\[
	C^{-1} \Vert Df\Vert(\Om)
	\le \liminf_{r\searrow 0}\int_\Om 
	\left[\,\frac{1}{\mu(B(y,r))}\int_{B(y,r)\cap\Om}\frac{|f(x)-f(y)|^q}{d(x,y)^q}\,d\mu(x)\right]^{1/q}\,d\mu(y).
	\]
	If $\Om\subset X$ is a bounded $q$-extension domain and $f\in \widehat{N}^{1,q}(\Om)$, then
	\[
	\limsup_{r\searrow 0}\int_\Om 
	\left[\,\frac{1}{\mu(B(y,r))}\int_{B(y,r)\cap\Om}\frac{|f(x)-f(y)|^q}{d(x,y)^q}\,d\mu(x)
	\right]^{1/q}\,d\mu(y)
	\le C \Vert Df\Vert(\Om).
	\]
	Here $C\ge 1$ is a constant depending only on $q$, the doubling constant of the measure,
	and the constants in the Poincar\'e inequality.
\end{corollary}

Finally, we discuss one corollary of Theorem \ref{thm:main phi}.
In \cite{BN2}, Brezis and Nguyen gave three examples of $\varphi$ when studying the asymptotic behavior of $\Lambda_{p,\delta}(f,\Omega)$. Here, we only consider the most simple one, that is, $\varphi(t)=0$ for $t\in [0,1]$ and $\varphi(t)=1$ for $t>1$. Applying Theorem \ref{thm:main phi} with this $\varphi$, we obtain the following result.  

\begin{corollary}\label{cor:phi}
	Suppose $\mu$ is Ahlfors $Q$-regular with $1<Q<\infty$,
	$\Om\subset X$ is open and bounded, and let $f\in L^1(\Om)$. Then
	\[
		C^{-1}\Vert Df\Vert(\Om)
		\le \limsup_{\delta\to 0}\int_{\Om}\int_{\{x\in \Om\colon |f(x)-f(y)|>\delta\}} \frac{\delta}{d(x,y)^{Q+1}} 
		\,d\mu(x)\,d\mu(y).
	\]
for some constant $C\ge 1$ depending only on the Ahlfors regularity constants.
\end{corollary}

\begin{remark}
Let $f\in \widehat{N}^{1,p}(X)$, with $1<p<\infty$. In \cite{BP} it is proved that the inequality in Corollary \ref{cor:phi} becomes an equality
in the setting of metric measure spaces endowed with a doubling measure, supporting
a $(1,p)-$Poincar\'e inequality, and such that at $\mu$-a.e. point the tangent space
(in the Gromov-Hausdorff sense) is unique and Euclidean with a fixed dimension. 
\end{remark}

\section{Lower bounds}\label{sec:lower bounds}

In this section we prove the lower bounds of our three main theorems.
As before, $\Om$ is an open subset of $X$.

\subsection{Lower bound of  Theorem \ref{thm:main eps}}

The following theorem proves the lower bound of Theorem \ref{thm:main eps};
the proof is rather standard and follows mostly as in \cite{BN1}.

\begin{theorem}
	Let $1\le p<\infty$ and
	suppose $\{\rho_i\}_{i=1}^{\infty}$ is a sequence of mollifiers satisfying conditions
	\eqref{eq:rho hat minorize}, \eqref{eq:rho hat majorize}.
	Let $f\in L^p(\Om)$. Then
	\[
	C_1' E_{p}(f,\Om) \le \liminf_{i\to\infty}\Psi_{p,i}(f,\Om),
	\]
	where $C_1'$ is a constant depending only on
	the doubling constant of the measure and the constant $C_{\rho}$ associated with the mollifiers.
\end{theorem}
\begin{proof}
	We can assume that $\Om$ is nonempty.
	Consider a nonempty bounded open set $U\subset \Om$.
	We apply H\"older's inequality with the exponents
	\[
	\frac{p+\eps_i}{p}\quad\textrm{and}\quad \frac{p+\eps_i}{\eps_i}
	\]
	to obtain
	\begin{align*}
		&\int_{U}\int_{U} \frac{|f(x)-f(y)|^p}{d(x,y)^p}\rho_i(x,y)\,d\mu(x)\,d\mu(y)\\
		&\le \left(\int_{U}\int_{U} \frac{|f(x)-f(y)|^{p+\eps_i}}{d(x,y)^{p+\eps_i}}\rho_i(x,y)\,d\mu(x)\,d\mu(y)
		\right)^{p/(p+\eps_i)}\\
		&\qquad \times \left(\int_U\int_U\rho_i(x,y)\,d\mu(x)\,d\mu(y)\right)^{\eps_i/(p+\eps_i)}\\
		&\le \left(\int_{U}\int_{U} \frac{|f(x)-f(y)|^{p+\eps_i}}{d(x,y)^{p+\eps_i}}\rho_i(x,y)\,d\mu(x)\,d\mu(y)
		\right)^{p/(p+\eps_i)}
		\left(C_{\rho}\mu(U)\right)^{\eps_i/(p+\eps_i)}
		\quad\textrm{by }\eqref{eq:rho hat majorize}.
	\end{align*}
	Thus
	\begin{equation}\label{eq:U initial estimate}
	\begin{split}
		&\left(C_{\rho}\mu(U)\right)^{-\eps_i/(p+\eps_i)}
		\int_{U}\int_{U} \frac{|f(x)-f(y)|^p}{d(x,y)^p}
		\rho_i(x,y)\,d\mu(x)\,d\mu(y)\\
		&\qquad \le \left(\int_{U}\int_{U} \frac{|f(x)-f(y)|^{p+\eps_i}}{d(x,y)^{p+\eps_i}}\rho_i(x,y)\,d\mu(x)\,d\mu(y)\right)^{p/(p+\eps_i)}.
	\end{split}
	\end{equation}
	We estimate
	\begin{align*}
		&\liminf_{i\to\infty}\left(\int_{\Om}\int_{\Om} \frac{|f(x)-f(y)|^{p+\eps_i}}{d(x,y)^{p+\eps_i}}\rho_i(x,y)\,d\mu(x)\,d\mu(y)\right)^{p/(p+\eps_i)}\\
		&\qquad \ge \liminf_{i\to\infty}\left(\int_{U}\int_{U} \frac{|f(x)-f(y)|^{p+\eps_i}}{d(x,y)^{p+\eps_i}}\rho_i(x,y)\,d\mu(x)\,d\mu(y)\right)^{p/(p+\eps_i)}\\
		&\qquad \ge \liminf_{i\to\infty}\int_{U}\int_{U} \frac{|f(x)-f(y)|^p}{d(x,y)^p}\rho_i(x,y)\,d\mu(x)\,d\mu(y)
		\quad\textrm{by }\eqref{eq:U initial estimate}\\
		&\qquad \ge C_1 E_{p}(f,U)\quad\textrm{by Theorem }
		\ref{thm:main previous}.
	\end{align*}
Since this holds for every bounded open $U\subset \Om$,
using the measure property of $\Vert Df\Vert$ in the case $p=1$ and 
\cite[Lemma 2.23]{BB} in the case $1<p<\infty$,
we obtain
\[
\liminf_{i\to\infty}\left(\int_{\Om}\int_{\Om} \frac{|f(x)-f(y)|^{p+\eps_i}}{d(x,y)^{p+\eps_i}}\rho_i(x,y)\,d\mu(x)\,d\mu(y)\right)^{p/(p+\eps_i)}
\ge C_1 E_{p}(f,\Om).
\]
\end{proof}

\subsection{Lower bound of  Theorem \ref{thm:main p}}

Given a ball $B=B(x,r)$ with a specified center $x\in X$ and radius $r>0$, we denote
$2B:=B(x,2r)$.
The distance between two sets $A,D\subset X$ is denoted by
\[
\dist(A,D):=\inf\{d(x,y)\colon x\in A,\,y\in D\}.
\]

The next lemma is standard; for a proof see \cite[Lemma 5.1]{LPZ}.
Recall the definition \eqref{eq:neighborhood notation}.

\begin{lemma}\label{lem:covering lemma}
	Consider an open set $U\subset \Om$ with $\dist(U,X\setminus \Om)>0$,
	and a scale $0<s<\dist(U,X\setminus \Om)/10$.
	Then we can choose an at most countable covering
	$\{B_j=B(x_j,s)\}_{j}$ of $U(5s)$ such that $x_j\in U(5s)$,
	each ball $5B_j$ is contained in $\Om$, and the balls $\{5B_j\}_{j=1}^{\infty}$ can be divided
	into $C_d^8$ collections of pairwise disjoint balls.
\end{lemma}

Given such a covering of $U(5s)$, 
we can take a partition of unity $\{\phi_j\}_{j=1}^{\infty}$ subordinate to the
covering, such that $0\le \phi_j\le 1$,
\begin{equation}\label{eq:Lipschitz function}
	\textrm{each } \phi_j  \textrm{ is a }3C_d^8/s\textrm{-Lipschitz function},
\end{equation}
and $\supp(\phi_j)\subset 2B_j$ for each 
$j\in\N$;
see e.g. \cite[p. 104]{HKST}.
Finally, we can define a \emph{discrete convolution} $h$ of 
any $f\in L^1(\Om)$ with respect to the covering by
\[
h:=\sum_{j}f_{B_j}\phi_j.
\]
Clearly $h\in \Lip_{\loc}(U)$.

Now we prove the lower bound of  Theorem \ref{thm:main p}.

\begin{theorem}\label{thm:lower bound}
	Let $1\le p<\infty$, $1\le q<\infty$, and suppose $\rho_i$ is a sequence of mollifiers satisfying
	either \eqref{eq:rho hat minorize} or the same condition with $p$ replaced by $q$.
	Suppose $f\in L^p(\Om)$. Then
	\begin{equation}\label{eq:main theorem equation lower}
		C_1'' E_{p}(f,\Om)
		\le \liminf_{i\to\infty}\int_{\Om}\left[\int_{\Om} \left(\frac{|f(x)-f(y)|^p}{d(x,y)^p}\right)^q\rho_i(x,y)\,d\mu(x)\right]^{1/q}\,d\mu(y)
	\end{equation}
	for some constant $C_1''$ depending only on $p$, $q$, $C_{\rho}$,
	 and on the doubling constant of the measure.
\end{theorem}
Note that here we do not impose any conditions on the open set $\Om\subset X$.

\begin{proof}
	We can assume that
	\[
	\liminf_{i\to\infty}\int_{\Om}\left[\int_{\Om} \left(\frac{|f(x)-f(y)|^p}{d(x,y)^p}\right)^q\rho_i(x,y)\,d\mu(x)\right]^{1/q}\,d\mu(y)=:M<\infty.
	\]
	Fix $0<\eps<1$. Passing to a subsequence (not relabeled), we can assume that
	\[
	\int_{\Om}\left[\int_{\Om}
	\left(\frac{|f(x)-f(y)|^p}{d(x,y)^p}\right)^q\rho_i(x,y)\,d\mu(x)\right]^{1/q}\,d\mu(y)\le M+\eps \quad\textrm{for all }i\in\N.
	\]
	Using the assumption \eqref{eq:rho hat minorize}
	(if we have this condition with $p$ replaced by $q$, the next three lines are similar), we get
	\[
		\int_{\Om} \left[\int_{\Om} \frac{|f(x)-f(y)|^{pq}}{r_i^p d(x,y)^{pq-p}}
		\frac{\ch_{B(y,r_i)\cap \Om}(x)}{\mu(B(y,r_i))}\,d\mu(x)\right]^{1/q}\,d\mu(y)
		\le (M+\eps)C_{\rho}^{1/q},
	\]
	and thus
	\begin{equation}\label{eq:t arbitrarily small}
		\int_{\Om} \left[\int_{\Om} \frac{|f(x)-f(y)|^{pq}}{r_i^{pq}}
		\frac{\ch_{B(y,r_i)\cap \Om}(x)}{\mu(B(y,r_i))}\,d\mu(x)\right]^{1/q}\,d\mu(y)
		\le (M+\eps)C_{\rho}^{1/q}.
	\end{equation}
	Fix $i\in\N$.
	Let $U\subset \Om$ be open with $\dist(U,X\setminus \Om)>r_i$, and
	let $s:=r_i/10$.
	Consider a covering $\{B_j\}_{j=1}^{\infty}$ of $U(5s)$ at scale $s>0$,
	as described in Lemma \ref{lem:covering lemma}. Then consider the discrete convolution
	\[
	h:=\sum_{j}f_{B_j}\phi_j.
	\]
	Recall the definition of the Lipschitz number $\Lip h$ from \eqref{eq:Lip}.
	Suppose $x\in U$. Then $x\in B_j$ for some $j\in\N$. Consider any other point $y\in B_j$.
	Denote by $I_j$ those $k\in\N$ for which $2B_k \cap B_j\neq \emptyset$.
	We estimate
	\begin{equation}\label{eq:upper gradient of discrete convolution}
		\begin{split}
			|h(x)-h(y)|
			& = \left|\sum_{k\in I_j}f_{B_k}(\phi_k(x)-\phi_k(y))\right|\\
			& = \left|\sum_{k\in I_j}(f_{B_k}-f_{B_j})(\phi_k(x)-\phi_k(y))\right|\\
			&\le \frac{3C_d^8 d(x,y)}{s}\sum_{k\in I_j}\,|f_{B_k}-f_{B_j}|\quad\textrm{by }\eqref{eq:Lipschitz function}\\
			&\le \frac{3C_d^8 d(x,y)}{s}\left(\sum_{k\in I_j}\,\vint{B_k}|f-f_{5B_j}|\,d\mu
			+\sum_{k\in I_j}\,\vint{B_j}|f-f_{5B_j}|\,d\mu\right)\\
			&\le \frac{6C_d^{11} d(x,y)}{s}\sum_{k\in I_j}\,\vint{5B_j}|f-f_{5B_j}|\,d\mu\\
			&\le \frac{6 C_d^{19} d(x,y)}{s}\,\vint{5B_j}\,\vint{5B_j}|f(z)-f(w)|\,d\mu(z)\,d\mu(w),
		\end{split}
	\end{equation}
	since by Lemma \ref{lem:covering lemma} we know that $I_j$ has cardinality at most $C_d^8$.
	Letting $y\to x$, we obtain an estimate for $\Lip_h$ in the ball $B_j$.
	In total, we conclude (we track the constants for a while in order to make the estimates more explicit)
	that in $U$ it holds that
	\[
	\Lip h
	\le \frac{6 C_d^{19}}{s}\sum_{j}\,\ch_{B_j}\vint{5B_j}\,\vint{5B_j}|f(x)-f(y)|\,d\mu(x)\,d\mu(y).
	\]
	Now
	\begin{align*}
			\Lip h
			&\le \frac{6 C_d^{19}}{s}
			\sum_{j}\ch_{B_j}\vint{5B_j}\,\vint{5B_j}|f(x)-f(y)|\,d\mu(x)\,d\mu(y)\\ 
			&\le \frac{6 C_d^{19}}{s}
			\sum_{j}\ch_{B_j}\left[\,\vint{5B_j}\left[\,\vint{5B_j}|f(x)-f(y)|^{pq}\,d\mu(x)\right]^{1/q}\,d\mu(y)\right]^{1/p}
			\quad\textrm{by H\"older}\\ 
			&\le \frac{6 C_d^{21}}{s}
			\sum_{j}\ch_{B_j}\left[\,\vint{5B_j}\left[\int_{5B_j}|f(x)-f(y)|^{pq}\frac{\ch_{B(y,10s)}(x)}
			{\mu(B(y,10s))}\,d\mu(x)\right]^{1/q}\,d\mu(y)\right]^{1/p}.
		\end{align*}
	Since the balls  $\{5B_j\}_{j}$ and so also the balls $\{B_j\}_{j}$ can be divided
	into $C_d^8$ collections of pairwise disjoint balls, in $U$ it holds that
	\[
	(\Lip h)^p
	\le \left(\frac{6 C_d^{29}}{s}\right)^p
	\sum_{j}\vint{5B_j}\left[\int_{5B_j}|f(x)-f(y)|^{pq}\frac{\ch_{B(y,10s)}(x)}
	{\mu(B(y,10s))}\,d\mu(x)\right]^{1/q}\,d\mu(y),
	\]
	and so
	\begin{equation}\label{eq:Lip h estimate}
		\begin{split}
			\int_U (\Lip h)^p\,d\mu
			&\le \left(\frac{6 C_d^{29}}{s}\right)^p
			\sum_{j}\int_{5B_j}\left[\int_{5B_j}|f(x)-f(y)|^{pq}\frac{\ch_{B(y,10s)}(x)}
			{\mu(B(y,10s))}\,d\mu(x)\right]^{1/q}\,d\mu(y)\\
			&\le \left(\frac{6 C_d^{37}}{s}\right)^p\int_{\Om}\left[\int_{\Om}|f(x)-f(y)|^{pq}\frac{\ch_{B(y,10s)\cap \Om}(x)}
			{\mu(B(y,10s))}\,d\mu(x)\right]^{1/q}\,d\mu(y)\\
			&\le C(M+\eps)C_{\rho}^{1/q}
		\end{split}
	\end{equation}
	by \eqref{eq:t arbitrarily small}, with $C:=(60 C_d^{37})^p$.
	We know that the minimal $p$-weak upper gradient $g_{h}$ of $h$ in $U$ satisfies $g_{h}\le \Lip h$ $\mu$-a.e. in $U$,
	see e.g. \cite[Proposition 1.14]{BB}.
	
	Recall that we can do the above for each $r_i$, and that for a fixed $i$ we denoted
	$s=r_i/10$.
	From now on, we can consider any open $U\subset \Om$ with $\dist(U,X\setminus \Om)>0$.
	We get a sequence of discrete convolutions $\{h_i\}_{i=1}^{\infty}$ corresponding
	to scales $s_i\searrow 0$, such that $\{g_{h_i}\}_{i=1}^{\infty}$ is a bounded sequence in
	$L^p(U)$. From the properties of discrete convolutions, see e.g. \cite[Lemma 5.3]{HKT},
	we know that $h_i\to f$ in $L^p(U)$.\\
	
	\textbf{Case $p=1$:}\\
	We get
	\[
	\Vert Df\Vert(U)
	\le \liminf_{i\to\infty}\int_U g_{h_i}\,d\mu
	\le \liminf_{i\to\infty}\int_U \Lip h_i\,d\mu
	\le  C(M+\eps)C_{\rho}^{1/q},
	\]
	and so $f\in \BV(U)$.
	Note that $\Vert Df\Vert$ is a Radon measure on $\Om$.
	Exhausting $\Om$ by sets $U$, we obtain
	\begin{align*}
		\Vert Df\Vert(\Om)
		&\le C(M+\eps)C_{\rho}^{1/q}\\
		&= CC_{\rho}^{1/q}\left(	\liminf_{i\to\infty}\int_{\Om}\left[\int_{\Om} \left(\frac{|f(x)-f(y)|}{d(x,y)}\right)^q\rho_i(x,y)\,d\mu(x)\right]^{1/q}\,d\mu(y)+\eps\right).
	\end{align*}
	Letting $\eps\to 0$, this proves \eqref{eq:main theorem equation lower}.\\
	
	\textbf{Case $1<p<\infty$:}\\
	By \eqref{eq:Lip h estimate}, $\{g_{h_i}\}_{i=1}^{\infty}$ is a bounded sequence in $L^p(U)$.
	By reflexivity of the space $L^p(U)$, we find a subsequence of $\{g_{h_i}\}_{i=1}^{\infty}$ (not relabeled)
	and $g\in L^p(U)$ such that $g_{h_i}\to g$ weakly in $L^p(U)$ (see e.g. \cite[Section 2]{HKST}).
	By Mazur's lemma (Theorem \ref{thm:Mazur lemma}), for suitable convex combinations we get the
	strong convergence $\sum_{l=i}^{N_i}a_{i,l} g_{h_l}\to g$ in $L^p(U)$.
	Note that also $\sum_{l=i}^{N_i}a_{i,l} h_l\to f$ in $L^p(U)$.
	Using e.g. \cite[Proposition 2.3]{BB}, we know that there exists a function $\widehat{f}=f$
	$\mu$-a.e. in $U$ such that $g$ is a $p$-weak upper gradient of $\widehat{f}$ in $U$.
	We get
	\[
	E_{p}(f,U)
	\le \int_U g^p\,d\mu
	\le \limsup_{i\to\infty}\int_U (g_{h_i})^p\,d\mu
	\le \limsup_{i\to\infty}\int_U (\Lip h_i)^p\,d\mu
	\le  C(M+\eps)C_{\rho}^{1/q}.
	\]
	Exhausting $\Om$ by sets $U$ and using \cite[Lemma 2.23]{BB}, we obtain
	\begin{align*}
		E_{p}(f,\Om)
		&\le C(M+\eps)C_{\rho}^{1/q}\\
		&\le CC_{\rho}^{1/q}\left(	\liminf_{i\to\infty}\int_{\Om}\left[\int_{\Om} \left(\frac{|f(x)-f(y)|^p}{d(x,y)^p}\right)^q\rho_i(x,y)\,d\mu(x)\right]^{1/q}\,d\mu(y)+\eps\right).
	\end{align*}
	Letting $\eps\to 0$, this proves \eqref{eq:main theorem equation lower}.
\end{proof}

\subsection{Lower bound of  Theorem \ref{thm:main phi}}

The following theorem proves the lower bound of Theorem \ref{thm:main phi}.

\begin{theorem}\label{thm:lower bound phi}
	Let $1\le p<\infty$ and suppose $\varphi\colon [0,\infty)\to [0,\infty)$ satisfies assumptions
	\eqref{eq:varphi assumption increasing}
	and \eqref{eq:varphi integral assumption}.
	Suppose $\mu$ is Ahlfors $Q$-regular,
	$\Om$ is bounded,  and let $f\in L^p(\Om)$. Then
	\[
		E_{p}(f,\Om)
		\le C\limsup_{\delta\to 0}\int_{\Om}\int_{\Om} \frac{\delta^p\varphi(|f(x)-f(y)|/\delta)}{d(x,y)^{Q+p}} 
		\,d\mu(x)\,d\mu(y),
	\]
	where the constant $C$  depends only on $p$, $C_{\varphi}$, and the Ahlfors regularity constants. 
\end{theorem}

This proof is very similar to that in the Euclidean setting,
see \cite[Proposition 2]{BN18}.

\begin{proof}
We can assume that
\[
\limsup_{\delta\to 0}\int_{\Om}\int_{\Om} \frac{\delta^p\varphi(|f(x)-f(y)|/\delta)}{d(x,y)^{Q+p}} 
\,d\mu(x)\,d\mu(y)=:M<\infty.
\]
First suppose that $f\in L^{\infty}(\Om)$.
Fix $\eps>0$.
Also fix $\tau>0$ and, recalling \eqref{eq:varphi integral assumption},
choose $\delta_0>0$ sufficiently small that
\begin{equation}\label{eq:delta 0 choice}
	\int_{\delta_0}^{\infty}\varphi(t)t^{-1-p}\,dt\ge C_\varphi^{-1}(1-\tau)
\end{equation}
and
\begin{equation}\label{eq:delta 0 choice 2}
\Lambda_{p,\delta}(f,\Om)\le \limsup_{\delta\to 0}\Lambda_{p,\delta}(f,\Om)+\tau\quad\textrm{for all }0<\delta<\delta_0.
\end{equation}
We have by Fubini's theorem
\begin{align*}
	&C_A\int_0^{\delta_0}\eps\delta^{\eps-1}\Lambda_{p,\delta}(f,\Om)\,d\delta\\
	&\ge \int_0^{\delta_0}\eps\int_{\Om}\int_{\Om} 
	\frac{\delta^{\eps-1+p}\varphi(|f(x)-f(y)|/\delta)}{d(x,y)^{Q+p}} 
	\,d\mu(x)\,d\mu(y)
	\,d\delta\\
	&=\int_{\Om}\int_{\Om}\frac{\eps}{d(x,y)^{Q+p}}
	\int_0^{\delta_0}\delta^{\eps-1+p}
	\varphi(|f(x)-f(y)|/\delta)
	\,d\delta\,d\mu(x)\,d\mu(y)\\
	&=\int_\Om\int_\Om\frac{\eps |f(x)-f(y)|^{p+\eps}}{d(x,y)^{Q+p}}
	\int_{|f(x)-f(y)|/\delta_0}^{\infty}\varphi(t)t^{-1-p-\eps}\,dt\,d\mu(x)\,d\mu(y).
\end{align*}
Thus
\begin{align*}
	&C_A\int_0^{\delta_0}\eps\delta^{\eps-1}\Lambda_{p,\delta}(f,\Om)\,d\delta\\
	&\ge \int_\Om\int_{\{x\in\Om\colon |f(x)-f(y)|<\delta_0^2\}}\frac{\eps |f(x)-f(y)|^{p+\eps}}{d(x,y)^{Q+p}}
	\int_{\delta_0}^{\infty}\varphi(t)t^{-1-p-\eps}\,dt\,d\mu(x)\,d\mu(y),
\end{align*}
and then
\begin{equation}\label{eq:eps delta estimate}
	\begin{split}
		&C_A\int_0^{\delta_0}\eps\delta^{\eps-1}\Lambda_{p,\delta}(f,\Om)\,d\delta\\
		&\quad +\int_\Om\int_{\{x\in\Om\colon |f(x)-f(y)|\ge \delta_0^2\}}
		\frac{\eps (2\Vert f\Vert_{L^{\infty}(\Om)})^{p+\eps}}{d(x,y)^{Q+p}}\,d\mu(x)\,d\mu(y)
		\int_{\delta_0}^{\infty}\varphi(t)t^{-1-p-\eps}\,dt\\
		&\qquad \ge \int_\Om\int_{\Om}\frac{\eps |f(x)-f(y)|^{p+\eps}}{d(x,y)^{Q+p}}\,d\mu(x)\,d\mu(y)
		\int_{\delta_0}^{\infty}\varphi(t)t^{-1-p-\eps}\,dt.
	\end{split}
\end{equation}
For $\alpha>0$, by the fact that $\varphi$ is increasing we have
\begin{align*}
	\Lambda_{p,\delta}(f,\Om)
	&\ge \int_\Om\int_{\{x\in\Om\colon |f(x)-f(y)|\ge \alpha\}}
	\frac{\delta^p\varphi(|f(x)-f(y)|/\delta)}{d(x,y)^{Q+p}}\,d\mu(x)\,d\mu(y)\\
	&\ge \delta^p \varphi(\alpha/\delta)\int_\Om\int_{\{x\in\Om\colon |f(x)-f(y)|\ge \alpha\}}
	\frac{1}{d(x,y)^{Q+p}}\,d\mu(x)\,d\mu(y).
\end{align*}
Choosing $\delta$ sufficiently small such that $\varphi(\alpha/\delta)>0$, we conclude
that for every $\alpha>0$, we have
\[
\int_\Om\int_{\{x\in\Om\colon |f(x)-f(y)|\ge \alpha\}}
\frac{1}{d(x,y)^{Q+p}}\,d\mu(x)\,d\mu(y)<\infty.
\]
This combined with \eqref{eq:delta 0 choice}, \eqref{eq:eps delta estimate},
and Corollary \ref{cor:epsilon consequence} gives
\[
C_A \liminf_{\eps\to 0}\int_0^{\delta_0}\eps\delta^{\eps-1}\Lambda_{p,\delta}(f,\Om)\,d\delta
\ge (1-\tau)C_\varphi^{-1}C^{-1} E_{p}(f,\Om).
\]
On the other hand,
\begin{align*}
	(\limsup_{\delta\to 0}\Lambda_{p,\delta}(f,\Om)+\tau )\delta_0^{\eps}
	&=(\limsup_{\delta\to 0}\Lambda_{p,\delta}(f,\Om)+\tau )\int_0^{\delta_0}\eps(\delta')^{\eps-1}\,d\delta' \\
	& \ge \int_0^{\delta_0}\eps\delta^{\eps-1}\Lambda_{p,\delta}(f,\Om)\,d\delta
	\quad\textrm{by }\eqref{eq:delta 0 choice 2}.
\end{align*}
Thus
\[
\limsup_{\delta\to 0}\Lambda_{p,\delta}(f,\Om)+\tau	\ge (1-\tau)C_\varphi^{-1}C^{-1}C_A^{-1} E_{p}(f,\Om).
\]
Letting $\tau\to 0$, we get
\[
\limsup_{\delta\to 0}\Lambda_{p,\delta}(f,\Om)\ge C_\varphi^{-1}C^{-1} C_A^{-1} E_{p}(f,\Om).
\]

Finally, we drop the assumption $f\in L^{\infty}(\Om)$.
Consider the truncations
\[
f_M:=\min\{M,\max\{-M,f\}\},\quad M>0.
\]
Since $\varphi$ is  increasing, we have
\[
\limsup_{\delta\to 0}\Lambda_{p,\delta}(f,\Om)
\ge\limsup_{\delta\to 0}\Lambda_{p,\delta}(f_M,\Om)
\ge C_\varphi^{-1}C^{-1} C_A^{-1} E_{p}(f_M,\Om).
\]
Letting $M\to\infty$, by using the coarea formula \eqref{eq:coarea} in the case $p=1$,
and e.g. \cite[Proposition 2.3]{BB} in the case $1<p<\infty$, we get
\[
\limsup_{\delta\to 0}\Lambda_{p,\delta}(f,\Om)
\ge C_\varphi^{-1}C^{-1} C_A^{-1} E_{p}(f,\Om).
\]
\end{proof}

\section{Upper bounds}\label{sec:upper bounds}
In this section we prove the upper bounds of our three main theorems.
As before, $\Om$ is an open subset of $X$.

\subsection{Upper bound of Theorem \ref{thm:main eps}}

The following theorem proves the upper bound of Theorem \ref{thm:main eps}.

\begin{theorem}
	Suppose  $1\le p<\infty$, $p<q<\infty$, and $\{\rho_i\}_{i=1}^{\infty}$ is a sequence of mollifiers satisfying condition 
	\eqref{eq:rho hat majorize}.
	Suppose also that $X$ supports a $(1,p)$--Poincar\'e inequality,
	$\Om\subset X$ is a bounded strong $p$-extension domain and strong $q$-extension domain, and
	$f\in \widehat{N}^{1,q}(\Om)$. Then
	\[
	\limsup_{i\to\infty}\Psi_{p,i}(f,\Om)
	\le C_2' E_{p}(f,\Om),
	\]
	where $C_2'$ is a constant
	depending only on $p$,
	the doubling constant of the measure, the constants in the
	Poincar\'e inequality, and the constant $C_{\rho}$ associated with the mollifiers.
\end{theorem}
\begin{proof}
	The space supports a $(p,p)$--Poincar\'e as well as a $(q,q)$--Poincar\'e inequality,
	see e.g. \cite[Theorem 4.21]{BB}.
	For sufficiently large $i$ such that $p+\eps_i<q$,
	we apply H\"older's inequality with the exponents
	\[
	\frac{q-p}{q-p-\eps_i}\quad\textrm{and}\quad \frac{q-p}{\eps_i}
	\]
	to obtain
	\begin{align*}
		&\left(\int_{\Om}\int_{\Om} \frac{|f(x)-f(y)|^{p+\eps_i}}{d(x,y)^{p+\eps_i}}\rho_i(x,y)\,d\mu(x)\,d\mu(y)
		\right)^{p/(p+\eps_i)}\\
		&=\left(\int_{\Om}\int_{\Om} \left(\frac{|f(x)-f(y)|}{d(x,y)}\right)^{\frac{pq-p^2-p\eps_i}{q-p}+\frac{q\eps_i}{q-p}}
		\rho_i(x,y)^{\frac{q-p-\eps_i}{q-p}+\frac{\eps_i}{q-p}}
		\,d\mu(x)\,d\mu(y)
		\right)^{p/(p+\eps_i)}\\
		& \qquad \le \left(\int_{\Om}\int_{\Om} \frac{|f(x)-f(y)|^p}{d(x,y)^p}\rho_i(x,y)\,d\mu(x)\,d\mu(y)
		\right)^{(q-p-\eps_i)p/((q-p)(p+\eps_i))}\\
		&\qquad \qquad \times\left(\int_{\Om}\int_{\Om} \frac{|f(x)-f(y)|^{q}}{d(x,y)^{q}}\rho_i(x,y)\,d\mu(x)\,d\mu(y)
		\right)^{\eps_i p/((q-p)(p+\eps_i))}.
	\end{align*}
	Applying Theorem \ref{thm:main previous} to both factors, we get
	\begin{align*}
		&\limsup_{i\to\infty}\left(\int_{\Om}\int_{\Om} \frac{|f(x)-f(y)|^{p+\eps_i}}{d(x,y)^{p+\eps_i}}\rho_i(x,y)\,d\mu(x)\,d\mu(y)
		\right)^{p/(p+\eps_i)}\\
		&\qquad \le C_2E_{p}(f,\Om)
		\times \left(C_2E_{q}(f,\Om)\right)^{0}\\
		&\qquad = C_2E_{p}(f,\Om).
	\end{align*}
\end{proof}

\subsection{Upper bound of Theorem \ref{thm:main p}}

\begin{lemma}\label{lem:weak ug coincide}
	Let $1\le p<\infty$ and $p<q<\infty$, and $X$ suppose supports a $(1,p)$-Poincar\'e inequality.
	Let $f\in N_{\loc}^{1,q}(X)$.
	Then the minimal $p$-weak upper gradient $g_{f,p}$ and the minimal $q$-weak upper gradient
	$g_{f,q}$ satisfy
	\[
	g_{f,p}\le g_{f,q}\le \widetilde{C}g_{f,p}\quad\textrm{ a.e.}
	\]
	for a constant $\widetilde{C}$
	depending only on the doubling constant of the measure and the constants in the Poincar\'e inequality.
\end{lemma}
\begin{proof}
	For the first inequality, see \cite[Proposition 2.44]{BB}.
	
	Then we prove the second inequality.
	For the case $1<p<\infty$, see \cite[Corollary A.9]{BB}.
	In the case $p=1$, we can follow the argument given in the proof of
	\cite[Corollary A.9]{BB}, but instead of \cite[Corollary A.8]{BB} we use 
	\cite[Proposition 4.26]{Che}.
\end{proof}

Now we can prove the upper bound of Theorem \ref{thm:main p}.

\begin{theorem}\label{thm:one direction}
	Let $1\le p<\infty$, $1<q<\infty$, and suppose $X$ supports a $(1,p)$-Poincar\'e inequality.
	Let $\Om\subset X$ be a bounded $pq$-extension domain,
	suppose $\{\rho_i\}_{i=1}^{\infty}$ is a sequence of mollifiers that satisfy
	\eqref{eq:rho hat majorize}, and let $f\in \widehat{N}^{1,pq}(\Om)$.
	Then
	\begin{equation}\label{eq:upper bound with U}
		\limsup_{i\to\infty}\int_{\Om}\left[\int_{\Om} \left(\frac{|f(x)-f(y)|^p}{d(x,y)^p}\right)^{q} 
		\rho_{i}(x,y) \,d\mu(x)\right]^{1/q}\,d\mu(y)
		\le C E_{p}(f,\Om),
	\end{equation}
	where $C$ is a constant depending only on $p$, the doubling constant of the measure,
	the constants in the Poincar\'e inequality, and the constant $C_{\rho}$
	associated with the mollifiers.
\end{theorem}

\begin{proof}
	Since $\Om$ is a $pq$-extension domain, we can assume that $f\in \widehat{N}^{1,pq}(X)$.
	We can further assume that $f\in N^{1,pq}(X)$, since choice of a pointwise $\mu$-representative
	does not affect either side of \eqref{eq:upper bound with U}.
	Fix an upper gradient $g\in L^{pq}(X)$ of $f$.
	By H\"older's inequality, $X$ also supports a $(1,pq)$--Poincar\'e inequality.
	By Keith-Zhong \cite[Theorem 1.0.1]{KeZh},
	there exists $1<q'<q$ such that $X$ also supports a $(1,pq')$--Poincar\'e inequality.
	Choosing $q'$ sufficiently close to $q$, we then have that
	$X$ supports a $(pq,pq')$--Poincar\'e inequality; see \cite[Theorem 4.21]{BB}.
	Let us denote the constants of the $(pq,pq')$--Poincar\'e inequality by $C_P',\lambda'$;
	note that they only depend on the doubling constant of the measure and the constants $C_P,\lambda$
	of the original $(1,p)$-Poincar\'e inequality.
	
	Recall the definition of the Hardy--Littlewood maximal function from \eqref{eq:maximal function}.
	For every Lebesgue point $y\in\Om$ and every $r>0$, we estimate
	\begin{equation}\label{eq:telescope estimate}
		\begin{split}
			|f(y)-f_{B(y,r)}|
			&\le \sum_{k=1}^{\infty}|f_{B(y,2^{-k+1}r)}-f_{B(y,2^{-k}r)}|\\
			&\le C_d\sum_{k=1}^{\infty}\,\vint{B(y,2^{-k+1}r)}\big|f-f_{B(y,2^{-k+1}r)}\big|\,d\mu\\
			&\le C_P'C_d\sum_{k=1}^{\infty}2^{-k+1}r\left(\,\vint{B(y,2^{-k+1}\lambda' r)}g^{pq'}\,d\mu\right)^{1/(pq')}\\
			&\le 2C_P'C_dr \left(\mathcal M_{\lambda' r} g^{pq'}(y)\right)^{1/(pq')}.
		\end{split}
	\end{equation}
	For every $j\in\Z$
	and for every $y\in \Om$, we estimate
	\begin{align*}
		&\int_{X} \frac{|f(x)-f(y)|^{pq}}{d(x,y)^{pq}} \frac{\ch_{B(y,2^{j+1})\setminus B(y,2^{j})}(x)}{\mu(B(y,2^{j+1}))} \,d\mu(x)\\
		&\qquad \le 2^{pq}\int_{X} \frac{|f(x)-f_{B(y,2^{j+1})}|^{pq}}{d(x,y)^{pq}} \frac{\ch_{B(y,2^{j+1})\setminus B(y,2^{j})}(x)}{\mu(B(y,2^{j+1}))} \,d\mu(x)\\
		&\qquad \qquad+2^{pq}\int_{X} \frac{|f(y)-f_{B(y,2^{j+1})}|^{pq}}{d(x,y)^{pq}} \frac{\ch_{B(y,2^{j+1})\setminus B(y,2^{j})}(x)}{\mu(B(y,2^{j+1}))} \,d\mu(x)\\
		&\qquad \le \,2^{-jpq+pq}\vint{B(y,2^{j+1})} |f(x)-f_{B(y,2^{j+1})}|^{pq} \,d\mu(x)\\
		&\qquad \qquad+2^{-jpq+pq}\vint{B(y,2^{j+1})} |f(y)-f_{B(y,2^{j+1})}|^{pq}  \,d\mu(x)\\
		&\qquad \le 4^{pq} (C_P')^p \left(\vint{B(y,2^{j+1}\lambda' )} g(x)^{pq'} \,d\mu(x)\right)^{q/q'}
		+ (8C_d C_P')^{pq} \left(\mathcal M_{2^{j+1}\lambda'} g^{pq'}(y)\right)^{q/q'}\textrm{ by }\eqref{eq:telescope estimate}\\
		&\qquad\le 8^{1+pq} (C_d C_P')^{pq}\left(\mathcal M_{2^{j+1}\lambda'} g^{pq'}(y)\right)^{q/q'}.
	\end{align*}
	Using \eqref{eq:rho hat majorize}, we can now estimate
	\begin{equation}\label{eq:f and maximal function estimate}
		\begin{split}
			&\int_{\Om}\left[\int_{X} \frac{|f(x)-f(y)|^{pq}}{d(x,y)^{pq}} 
			\rho_{i}(x,y) \,d\mu(x)\right]^{1/q}\,d\mu(y)\\
			&\qquad\le \int_{\Om}\left[\sum_{j\in \Z}d_{i,j}\int_{X} \frac{|f(x)-f(y)|^{pq}}{d(x,y)^{pq}} 
			\frac{\ch_{B(y,2^{j+1})\setminus B(y,2^{j})}(x)}{\mu(B(y,2^{j+1}))} \,d\mu(x)\right]^{1/q}\,d\mu(y)\\
			&\qquad\le 8^{1+p} (C_d C_P')^p \int_{\Om}\left[\sum_{j\in \Z}d_{i,j}\left(\mathcal M_{2^{j+1}\lambda'}
			(g^{pq'}(y)\right)^{q/q'}\right]^{1/q}\,d\mu(y).
		\end{split}
	\end{equation}
	For every $i\in\N$, we can estimate
		\begin{align*}
			\sum_{j\in \Z}d_{i,j}\left(\mathcal M_{2^{j+1}\lambda'} g^{pq'}\right)^{q/q'}
			&\le \sum_{j\in \Z}d_{i,j}\left(\mathcal M g^{pq'}\right)^{q/q'}\\
			&\le C_{\rho} \left(\mathcal M g^{pq'}\right)^{q/q'}\\
			&\in L^1(X)
		\end{align*}
	by the Hardy--Littlewood maximal function theorem, see e.g. \cite[Theorem 3.5.6]{HKST}.
	Then also
	\begin{equation}\label{eq:L1 majorant}
		\left[\sum_{j\in \Z}d_{i,j}\left(\mathcal M_{2^{j+1}\lambda'}
	(g^{pq'}(y)\right)^{q/q'}\right]^{1/q}
	\le \left[C_{\rho} \left(\mathcal M g^{pq'}\right)^{q/q'}\right]^{1/q}
	\in L^1(\Om),
	\end{equation}
	since $\Om$ is bounded.
	Moreover, $\mathcal M_{2^{j+1}\lambda' } g^{pq'}(y)\to g^{pq'}(y)$ as $j\to -\infty$ for a.e. $y\in\Om$,
	and recall that
	\[
	\lim_{i\to\infty}\sum_{j\ge M} d_{i,j}=0
	\]
	for all $M\in\Z$, as given after \eqref{eq:rho hat majorize}.
	Thus
	\[
	\limsup_{i\to\infty}\sum_{j}d_{i,j}(\mathcal M_{2^{j+1}\lambda' } g^{pq'}(y))^{q/q'}\le C_\rho g^{pq}(y)
	\]
	for a.e. $y\in\Om$. From
	\eqref{eq:f and maximal function estimate} we get by an application 
	of Lebesgue's dominated convergence theorem, with the majorant given by \eqref{eq:L1 majorant},
	\[
	\limsup_{i\to\infty}\int_{\Om}\left[\int_{X} \frac{|f(x)-f(y)|^{pq}}{d(x,y)^{pq}} 
	\rho_{i}(x,y) \,d\mu(x)\right]^{1/q}\,d\mu(y)
	\le  8^{1+p} (C_d C_P')^p C_\rho^{1/q}\int_{\Om}g^p\,d\mu.
	\]
	Let $g_{f,pq}$ be the minimal $pq$-weak upper gradient of $f$,
	and let $g_{f,p}$ be the minimal $p$-weak upper gradient of $f$.
	Using e.g. \cite[Lemma 1.46]{BB}, we can find a sequence
	of upper gradients $g_l\ge g$ with $\Vert g_l-g_{f,pq}\Vert_{L^{pq}(X)}\to 0$,
	and then also $\Vert g_l-g_{f,pq}\Vert_{L^p(\Om)}\to 0$.
	Thus
	\begin{align*}
	\limsup_{i\to\infty}\int_{\Om}\left[\int_{\Om} \frac{|f(x)-f(y)|^{pq}}{d(x,y)^{pq}} 
	\rho_{i}(x,y) \,d\mu(x)\right]^{1/q}\,d\mu(y)
	&\le  8^{1+p} (C_d C_P')^p C_\rho^{1/q}\int_{\Om}g_{f,pq}^p\,d\mu\\
	&\le  8^{1+p} (C_d C_P')^p C_\rho^{1/q}\widetilde{C}^p\int_{\Om}g_{f,p}^p\,d\mu.
	\end{align*}
	by Lemma \ref{lem:weak ug coincide}.
	Finally, applying \eqref{eq:from BV to Sobolev} in the case  $p=1$, we get for all $1\le p<\infty$
	\[
	\limsup_{i\to\infty}\int_{\Om}\left[\int_{\Om} \frac{|f(x)-f(y)|^{pq}}{d(x,y)^{pq}} 
	\rho_{i}(x,y) \,d\mu(x)\right]^{1/q}\,d\mu(y)
	\le 8^{1+p} (C_d C_P')^p C_\rho^{1/p}C_*E_{p}(f,\Om).
	\]
	Note that the above constant $8^{1+p} (C_d C_P')^p C_\rho^{1/p}C_*$
	only depends on $C_d,C_P,\lambda,p,C_{\rho}$, as desired.
\end{proof}

\subsection{Upper bound of Theorem \ref{thm:main phi}}

In the following theorem we prove the upper bound of Theorem \ref{thm:main phi}.

\begin{theorem}\label{thm:phi upper bound U}
	Suppose $\varphi\colon [0,\infty)\to [0,\infty)$ satisfies assumptions
	\eqref{eq:varphi assumption increasing}, \eqref{eq:varphi assumption bounded},
	and \eqref{eq:varphi integral assumption}.
	Let $1\le p<\infty$,
	suppose $X$ supports a $(1,p)$--Poincar\'e inequality,
	let $\Om$ be bounded,
	and let
	$f\in \widehat{N}^{1,p}(X)$;
	in the case $p=1$ suppose additionally that $f\in \widehat{N}^{1,q}(X)$
	for some $q>1$.
	Then
	\[
	\limsup_{\delta\to 0}\int_{\Om}\int_{\Om} \frac{\delta^p\varphi(|f(x)-f(y)|/\delta)}{\mu(B(y,d(x,y)))d(x,y)^p} 
	\,d\mu(x)\,d\mu(y)
	\le C C_\varphi E_{p}(f,\Om),
	\]
	where $C$  depends  only on $p$, the doubling constant of the measure, the constants
	in the Poincar\'e inequality, and $C_{\varphi}$.
\end{theorem}

\begin{proof}
	When $1<p<\infty$, by Keith-Zhong \cite[Theorem 1.0.1]{KeZh}
	there exists $1<p'<p$ such that $X$ supports a $(1,p')$--Poincar\'e inequality.
	When $p=1$, we let also $p'=1$.
	For simplicity, we still denote the constants by $C_P,\lambda$.
	
	First note that using \eqref{eq:varphi assumption increasing} and
	\eqref{eq:varphi integral assumption}, for any $0<\alpha<\infty$ we get
	\begin{equation}\label{eq:integral sum estimate}
		C_{\varphi}\ge \int_0^{\infty}\varphi(t)t^{-1-p}\,dt
		\ge 2^{-1-p}\alpha^{-p}\sum_{j\in\Z}\varphi(2^j \alpha) 2^{-jp}.
	\end{equation}
	We can assume that $f\in N^{1,p}(X)$,
	and in the case $p=1$ we can assume that  $f\in N^{1,1}(X)\cap N^{1,q}(X)$;
	note that the choice of pointwise representatives does not cause problems, by e.g.
	\cite[Propositions 1.59 \& 1.61]{BB}.
	For a.e. $x,y\in X$, we have
	\[
	|f(x)-f(y)|\le C'd(x,y)
	[(\mathcal M_{2\lambda d(x,y)}g_f(x)^{p'})^{1/p'}+(\mathcal M_{2\lambda d(x,y)}g_f(y)^{p'})^{1/p'}],
	\]
	where $g_f$ is the minimal $p$-weak upper gradient of $f$ and
	$C'$ is a constant depending only on $C_d$, $C_P$, and $\lambda$;
	see e.g. \cite[Proof of Theorem 5.1]{BB}. 
	Using this and the fact that $\varphi$ is increasing, we estimate
	(we denote briefly $\{d(x,y)<r\}:=\{(x,y)\in \Om\times \Om\colon d(x,y)<r\}$)
	\begin{equation}\label{eq:byr estimate}
		\begin{split}
			&\iint_{\{ d(x,y)<r\}} \frac{\delta^p\varphi(|f(x)-f(y)|/\delta)}
			{\mu(B(y,d(x,y)))d(x,y)^p} \,d\mu(x)\,d\mu(y)\\
			& \le \iint_{\{d(x,y)<r\}} \frac{\delta^p \varphi\big(C' d(x,y)((\mathcal M_{2\lambda d(x,y)}g_f(x)^{p'})^{\frac{1}{p'}}+(\mathcal M_{2\lambda d(x,y)}g_f(y)^{p'})^{\frac{1}{p'}} /\delta\big)}{\mu(B(y,d(x,y)))d(x,y)^p} \,d\mu(x)\,d\mu(y)\\
			& \le \int_{\Om} \int_{B(y,r)}\frac{\delta^p \varphi\big(2C'd(x,y)
				(\mathcal M_{2\lambda r} g_f(y)^{p'})^{\frac{1}{p'}} /\delta\big)}{\mu(B(y,d(x,y)))d(x,y)^p} \,d\mu(x)\,d\mu(y)\\
			&\qquad +\int_{\Om} \int_{B(x,r)}
			\frac{\delta^p \varphi\big(2C'd(x,y)
				(\mathcal M_{2\lambda r} g_f(x)^{p'})^{\frac{1}{p'}} /\delta\big)}{\mu(B(y,d(x,y)))d(x,y)^p} \,d\mu(y)\,d\mu(x)
		\end{split}
	\end{equation}
	using again the fact that $\varphi$ is increasing.
	Note that
	\begin{align*}
		& \int_{B(y,r)}\frac{\delta^p \varphi\big(2C'd(x,y)
			(\mathcal M_{2\lambda r} g_f(y)^{p'})^{\frac{1}{p'}} /\delta\big)}{\mu(B(y,d(x,y)))d(x,y)^p} \,d\mu(x)\\
		&\qquad \le \int_{X} \frac{\delta^p \varphi\big(2C'd(x,y)
			(\mathcal M_{2\lambda r} g_f(y)^{p'})^{\frac{1}{p'}} /\delta\big)}{\mu(B(y,d(x,y)))d(x,y)^p} \,d\mu(x)\\
		& \qquad = \sum_{j\in\Z}\int_{B(y,2^j)\setminus B(y,2^{j-1})} \frac{\delta^p \varphi\big(2C'd(x,y)
			(\mathcal M_{2\lambda r} g_f(y)^{p'})^{\frac{1}{p'}} /\delta\big)}{\mu(B(y,d(x,y)))d(x,y)^p} \,d\mu(x)\\
		& \qquad  \le C_d \delta^p\sum_{j\in\Z} \frac{\varphi(2^{j+1} C'(\mathcal M_{2\lambda r} g_f(y)^{p'})^{\frac{1}{p'}} /\delta)}{2^{(j-1)p}}\\
		& \qquad  \le 2^{3p+1}(C')^pC_d C_\varphi(\mathcal M_{2\lambda r} g_f(y)^{p'})^{\frac{p}{p'}}
		\quad\textrm{by }\eqref{eq:integral sum estimate}.
	\end{align*}
	Estimating the second term of \eqref{eq:byr estimate} analogously, we get
	\[
	\iint_{\{d(x,y)<r\}} \frac{\delta^p\varphi(|f(x)-f(y)|/\delta)}
	{\mu(B(y,d(x,y)))d(x,y)^p} \,d\mu(x)\,d\mu(y)
	\le 2^{3p+2}(C')^pC_d C_\varphi\int_\Om (\mathcal M_{2\lambda r} g_f^{p'})^{\frac{p}{p'}}\,d\mu.
	\]
	On the other hand, we have
	\begin{equation}\label{eq:byr complement estimate}
		\begin{split}
			&\int_{\Om\setminus B(y,r)}\frac{\delta^p \varphi(|f(x)-f(y)|/\delta)}
			{\mu(B(y,d(x,y)))d(x,y)^p}\,d\mu(x)\\
			& \qquad =\sum_{j\in\N}\int_{B(y,2^j r)\setminus B(y,2^{j-1}r)}\frac{\delta^p \varphi(|f(x)-f(y)|/\delta)}
			{\mu(B(y,d(x,y)))d(x,y)^p}\,d\mu(x)\\
			&\qquad \le  C_d\sum_{j=1}^{\infty}\frac{\delta^p b}{(2^{j-1}r)^p}
			\quad\textrm{by }\eqref{eq:varphi assumption bounded}\\
			&\qquad =\frac{2C_d\delta^p b}{r^p}.
		\end{split}
	\end{equation}
	In total, we get
	\begin{equation}\label{eq:delta r terms}
		\begin{split}
			&\int_{\Om}\int_{\Om} \frac{\delta^p\varphi(|f(x)-f(y)|/\delta)}
			{\mu(B(y,d(x,y)))d(x,y)^p} \,d\mu(x)\,d\mu(y)\\
			&\qquad \le 2^{2p+3} (C')^pC_d C_\varphi\int_\Om (\mathcal M_{2\lambda r} g_f^{p'})^{\frac{p}{p'}}\,d\mu
			+\frac{2C_d\delta^p b}{r^p}\mu(\Om).
		\end{split}
	\end{equation}
	In the case $1<p<\infty$, note that $(\mathcal M_{2\lambda r} g_f^{p'})^{\frac{p}{p'}}\in L^1(X)$
	by the Hardy--Littlewood maximal theorem,
	and that $\mathcal M_{2\lambda r} g_f(x)\to g_f(x)$ as $r\to 0$ for $\mu$-a.e. $x\in X$.
	Thus by \eqref{eq:delta r terms} and Lebesgue's dominated convergence,
	\begin{align*}
		&\limsup_{\delta\to 0}\int_{\Om}\int_{\Om} \frac{\delta^p\varphi(|f(x)-f(y)|/\delta)}
		{\mu(B(y,d(x,y)))d(x,y)^p} \,d\mu(x)\,d\mu(y)\\
		&\qquad \le 2^{3p+2}(C')^pC_d C_\varphi \int_\Om (\mathcal M_{2\lambda r} g_f^{p'})^{\frac{p}{p'}}\,d\mu\\
		&\qquad \to 2^{3p+2}(C')^pC_d C_\varphi \int_{\Om} g_f^p\,d\mu\quad\textrm{as }r\to 0\\
		&\qquad = 2^{3p+2}(C')^pC_d C_\varphi E_{p}(f,\Om).
	\end{align*}
	
	Then suppose $p=1$;
	recall that in this case we assume that	$f\in N^{1,1}(X)\cap N^{1,q}(X)$.
	By Lemma \ref{lem:weak ug coincide}, the minimal $1$-weak and minimal $q$-weak
	upper gradients of $f$ satisfy $g_{f,1}\le g_{f,q}$ a.e., and so in \eqref{eq:delta r terms}
	we can replace $g_{f,1}$ with $g_{f,q}$.
	Note that $\mathcal M_{2\lambda r} g_{f,q}\in L^q(X)$ by the Hardy--Littlewood maximal theorem,
	so that $\mathcal M_{2\lambda r} g_{f,q}\in L_{\loc}^1(X)$
	and in particular $\mathcal M_{2\lambda r} g_{f,q}\in L^1(\Om)$, and that
	$\mathcal M_{2\lambda r} g_{f,q}(x)\to g_{f,q}(x)$ as $r\to 0$ for $\mu$-a.e. $x\in X$.
	From \eqref{eq:delta r terms} and Lebesgue's dominated convergence, we get
	\begin{align*}
		&\limsup_{\delta\to 0}\int_{\Om}\int_{\Om} \frac{\delta\varphi(|f(x)-f(y)|/\delta)}
		{\mu(B(y,d(x,y)))d(x,y)} \,d\mu(x)\,d\mu(y)\\
		&\qquad \le 32  C' C_d C_\varphi\int_\Om \mathcal M_{2\lambda r} g_{f,q}\,d\mu\\
		&\qquad \to 32  C' C_d C_\varphi\int_\Om g_{f,q}\,d\mu\quad\textrm{as }r\to 0\\
		&\qquad \le 32  C' C_d C_\varphi \widetilde{C}\int_\Om g_{f,1}\,d\mu\quad\textrm{by Lemma }\ref{lem:weak ug coincide}\\
		&\qquad \le 32  C' C_d C_\varphi \widetilde{C} C_*\Vert Df\Vert(\Om)
	\end{align*}
	by \eqref{eq:from BV to Sobolev}.
\end{proof}

\section{Example}\label{sec:counterexample}

For a function $f\in \widehat{N}^{1,q}(\Om)$ with $1<q<\infty$,
and with the choice $p=1$,
the conclusion of Theorem \ref{thm:main p} takes the form
\begin{equation}\label{eq:recite main result}
	\begin{split}
		C_1'' \Vert Df\Vert(\Om)
		&\le \liminf_{i\to\infty}\int_{\Om}
		\left[\int_{\Om} \frac{|f(x)-f(y)|^q}{d(x,y)^q}\rho_i(x,y)\,d\mu(x)\right]^{1/q}\,d\mu(y)\\
		&\le \limsup_{i\to\infty}\int_{\Om}
		\left[\int_{\Om} \frac{|f(x)-f(y)|^q}{d(x,y)^q}\rho_i(x,y)\,d\mu(x)\right]^{1/q}\,d\mu(y)
		\le C_2'' \Vert Df\Vert(\Om).
	\end{split}
\end{equation}
One natural question is whether $C_1''=C_2''$ might hold.
In the Euclidean theory, see Brezis--Nguyen \cite[Proposition 9]{BN1},
we know that for every $1<q<\infty$, smooth and bounded $\Om\subset \R^n$,
and every $f\in W^{1,q}(\Om)$, we indeed have 
\begin{equation}\label{eq:Euclidean 1d result}
	\lim_{i\to\infty}\int_{\Om}\left[\int_{\Om} \frac{|f(x)-f(y)|^q}{|x-y|^q}\rho^*_i(|x-y|)\,
	dx\right]^{1/q}\,dy
	=K_{q,n}\Vert Df\Vert(\Om),
\end{equation}
where
\[
K_{q,n}:=\left(\int_{\mathbb{S}^{n-1}}|\sigma\cdot e|^q\,d\mathcal H^{n-1}(\sigma)\right)^{1/q},
\]
with $\mathbb{S}^{n-1}\subset \R^n$ is the unit sphere, $e\in \mathbb{S}^{n-1}$
is arbitrary, and
$\mathcal H^{n-1}$ is the $n-1$-dimensional Hausdorff measure.
Here the mollifiers $\rho^*_i$ are required to satisfy \eqref{eq:intro rhoi conditions}.
However, if $f$ is only a BV function, it may happen that \eqref{eq:Euclidean 1d result} fails,
while a (different) limit exists, see \cite[Proposition 10]{BN1}.
Many other counterexamples in the same spirit are given in e.g. \cite{BN1, BN18}.
In metric spaces things can go even more seriously wrong, as we will now demonstrate. 

Consider the real line equipped with the Euclidean metric and the one-dimensional Lebesgue
measure $\mathcal L^1$.
Consider the sequence of mollifiers
\[
\rho_i(x,y):=\rho^*_i(|x-y|):=\frac{\ch_{[0,1/i]}(|x-y|)}{1/i},\quad x,y\in \R,\quad i\in\N.
\]
As noted before Corollary \ref{cor:Brezis}, these mollifiers satisfy assumptions
\eqref{eq:rho hat minorize} and \eqref{eq:rho hat majorize}.
The mollifiers $\rho^*_i$ obviously also satisfy \eqref{eq:intro rhoi conditions}.
The example below, inspired by \cite[Example 4.8]{HKLL},
shows that even in a compact PI space and for a Lipschitz function $f$, 
 the desired equality $C_1''=C_2''$ in \eqref{eq:recite main result} may fail,
and so in particular the Euclidean result \eqref{eq:Euclidean 1d result} does not extend to metric spaces.

\begin{example}\label{ex:fat Cantor}
	Consider the space $X=[0,1]$, equipped with the Euclidean metric and a weighted measure
	$\mu$ that we will next define.
	First we construct a fat Cantor set $A$ as follows.
	Let $A_0:=[0,1]$.
	Then in each step
	$i\in\N$, we remove from $A_{i-1}$ the set $D_i$, which consists of $2^{i-1}$ open intervals of length $2^{-2i}$,
	centered at the middle points of the intervals that make up $A_{i-1}$.
	We denote $L_i:=\mathcal L^1(A_i)$,
	and we let $A=\bigcap_{i=1}^{\infty}A_i$. Then we have
	\[
	L:=\mathcal L^1(A) =\lim_{i\to\infty}L_i=1/2.
	\]
	Then define the weight
	\[
	w:=
	\begin{cases}
		2\quad\textrm{in }A,\\
		1\quad\textrm{in }X\setminus A,
	\end{cases}
	\]
	and equip the space $X$ with the weighted Lebesgue measure $d\mu:=w\,d\mathcal L^1$.
	Obviously the measure is doubling, and $X$ supports a $(1,1)$--Poincar\'e inequality.
	
	Let
	\[
	g:=2\ch_{A} \quad\textrm{and}\quad g_i=\frac{1}{L_{i-1}-L_i}\ch_{D_i},\ \ i\in\N.
	\]
	Then
	\[
	\int_0^1 g(s)\,ds= \int_0^1 g_i(s)\,ds=1\quad\textrm{for all }i\in\N.
	\]
	Next define the function
	\[
	f(x):=\int_0^x g(s)\, ds,\quad x\in [0,1].
	\]
	Now $f\in \Lip(X)$, since $g$ is bounded.
	Approximate $f$ with the functions
	\[
	f_i(x):=\int_0^x g_i(s)\, ds,\quad x\in [0,1],\quad i\in\N.
	\]
	Now also $f_i\in \Lip(X)$, and $f_i\to f$ uniformly.
	This can be seen as follows. Given $i\in\N$, the set $A_i$ consists of $2^i$ intervals of length
	$L_i/2^i$. If $I$ is one of these intervals, we have
	\[
	2^{-i}=\int_I g(s) \,ds =\int_I g_{i+1}(s) \,ds,
	\]
	and also
	\[
	\int_{X\setminus A_i}g\,d\mathcal L^1 =0=\int_{X\setminus A_i}g_{i+1}\,d\mathcal L^1 .
	\]
	Hence $f_{i+1}=f$ in  $X\setminus A_i$, and elsewhere $|f_{i+1}-f|$ is at most $2^{-i}$.
	In particular, $f_i\to f$ in $L^1(X)$ and so
	\[
	\Vert Df\Vert(X)\le \lim_{i\to\infty} \int_0^1 g_i\,d\mu
	= \lim_{i\to\infty} \int_0^1 g_i\,d\mathcal L^1=1.
	\]
	By Rademacher's theorem, for $\mathcal L^1$-a.e. $y\in A$, $f$ is differentiable at $y$ and so we have
	\[
	\lim_{i\to\infty}\left(\int_X \frac{|f(x)-f(y)|^q}{|x-y|^q}
	\rho_i(x,y)\,d\mathcal L^1(x)\right)^{1/q}=(2|f'(y)|^q)^{1/q}=2^{1/q}|f'(y)|.
	\]
	Thus
	\begin{align*}
		&\liminf_{i\to\infty}
		\int_X \left[\int_X\frac{|f(x)-f(y)|^q}{|x-y|^q}\rho_i(x,y)\,d\mu(x)\right]^{1/q}\,d\mu(y)\\
		&\qquad \ge 2 \liminf_{i\to\infty}\int_A \left[\int_X\frac{|f(x)-f(y)|^q}{|x-y|^q}\rho_i(x,y)\,d\mathcal L^1(x)\right]^{1/q}
		\,d\mathcal L^1(y)\\
		&\qquad \ge 2\int_A \liminf_{i\to\infty} \left[\int_X\frac{|f(x)-f(y)|^q}{|x-y|^q}\rho_i(x,y)\,d\mathcal L^1(x)\right]^{1/q}
		\,d\mathcal L^1(y)\quad\textrm{by Fatou}\\
		&\qquad= 2^{1+1/q}\int_A |f'(y)|\,d\mathcal L^1(y)\\
		&\qquad=2^{1+1/q}.
	\end{align*}
	We conclude that
	\begin{equation}\label{eq:comparison with constant 2}
		\liminf_{i\to\infty}
		\int_X \left[\int_X\frac{|f(x)-f(y)|^q}{|x-y|^q}\rho_i(x,y)\,d\mu(x)\right]^{1/q}\,d\mu(y)
		\ge 2^{1+1/q}\Vert Df\Vert(X).
	\end{equation}
	On the other hand, consider any nonzero Lipschitz function $f_0$ supported in $(3/8,5/8)$.
	For such a function, from \eqref{eq:Euclidean 1d result} we have
	\[
	\lim_{i\to\infty}\int_{X}\left[\int_{X} \frac{|f_0(x)-f_0(y)|^q}{|x-y|^q}\rho_i(x,y)\,d\mu(x)\right]^{1/q}\,d\mu(y)
	=2^{1/q}\Vert Df_0\Vert (X)\in (0,\infty),
	\]
	since both sides are equal to the classical quantities, that is,
	the quantities obtained when the measure $\mu$ is $\mathcal L^1$.
	This combined with \eqref{eq:comparison with constant 2} shows that we cannot have
	$C_1''=C_2''$ in \eqref{eq:recite main result}.
\end{example}

\end{document}